\documentclass[11pt]{article}
\usepackage{amsmath,amsthm,amsfonts,latexsym,amssymb,color}
\usepackage{a4}
\usepackage{xtab}
\usepackage{hyperref}
\usepackage{mathtools}
\usepackage[titletoc,title]{appendix}
\usepackage{todonotes}
\usepackage[driver=pdftex]{geometry}
\usepackage{footmisc}

\def\al{{\alpha}}
\def\be{{\beta}}
\def\ga{{\gamma}}

\def\th{{\theta}}

\def\cu{{\kappa}}

\def\om{{\omega}}

\def\ellog{{\mathfrak{l}}}

\def\C{{\mathbb C}}  
\def\R{{\mathbb R}}                                 
\def\Q{{\mathbb Q}}     
\def\Z{{\mathbb Z}}   
\def\N{{\mathbb N}}

\def\bx{{\mathrm{x}}} \def\by{{\mathrm{y}}} \def\bg{{\mathrm{g}}}
\def\A{{\mathrm A}}  
\def\B{{\mathrm B}}  
\def\C'{{\mathrm C}}  
\def\g{{\mathrm g}}  
\def\deg{{\mathrm{deg}}}
 \def\Res{{\mathrm{Res}}}

\newtheorem{proposition}{Proposition}[subsection]
\newtheorem{theorem}[proposition]{Theorem}
\newtheorem{corollary}[proposition]{Corollary}
\newtheorem{lemma}{Lemma}[section]
\newtheorem*{lemma*}{Lemma}

\begin{document}

\title{{Diophantine equations coming from binomial near-collisions.}}
\author{\large{Nikos Katsipis} 
   \\\normalsize{katsipis@gmail.com}\\ \normalsize{High School of Thira, Santorini, Greece}
\thanks{This work is part of the author's Doctoral Thesis at the
Department of Mathematics \& Applied Mathematics, University of Crete.}}
\date{} 

\maketitle

\section{{\large Introduction}}
               \label{sec intro}
Given a positive integer $d$ and a pair $(k,l)$ of unequal integers 
$\geq 2$, we say that there exists a 
\emph{(binomial) $(k,l)$ near-collision with difference $d$}
if there exists a pair $(m,n)$ of integers with    
$2\leq k\leq n/2$, $2\leq l\leq m/2$, such that
$\binom{m}{l}-\binom{n}{k}=d$ and $\binom{m}{l}\geq d^3$.
In such a case, the quadruple $(n,k,m,l)$ is said to be a
\emph{(binomial) near collision with difference $d$}. 
\\
Note that the above restrictions on $k,l$ are very natural in view of
the symmetries $\binom{m}{l}=\binom{m}{m-l}$ and
$\binom{n}{k}=\binom{n}{n-k}$. 
The rather arbitrary condition $\binom{m}{l}\geq d^3$ is just to ensure
that the difference between the two binomial coefficients is quite small 
compared to the greater one. As explained in \cite{dW et al}, it is
probably more natural to replace the exponent 3 of $d$ from the exponent 5. 

If we consider $k,l\geq 2$ and $d\neq 0$ (not-necessarily positive) 
as given fixed integers with $k\neq l$ we obtain the Diophantine equation
\begin{equation}
             \label{eq general collision}
\binom{m}{l}-\binom{n}{k} = d,             
\end{equation}
in the positive integer unknowns $m,n$,
without any restriction on the size of $\binom{m}{l}$ compared to $d$. 
In Section \ref{sec (k,l)=(3,6)} we will solve \eqref{eq general collision}
when $(k,l)=(3,6)$ for various values of $d$, and in 
Section \ref{sec (k,l)=(8,2)} we will solve \eqref{eq general collision} 
with $(k,l)=(8,2)$ and $d=1$.
Our main results, Theorems \ref{thm d=-1,main}, \ref{thm various d,main}, 
\ref{thm sols of quartic eq} respectively imply 
Corollaries \ref{corol (k,l)=(3,6)},\ref{corol (k,l)=(6,3)},
\ref{corol (k,l)=(2,8)}. As a consequence we have that 
\emph{
$(k,l)$-near collisions with difference 1 do not exist if 
$(k,l)\in\{(6,3),(3,6),(8,2)\}$}, establishing thus a conjecture stated in 
\cite[Section 3.1]{dW et al}.

We now sketch the method which we apply in Sections \ref{sec (k,l)=(3,6)} 
and \ref{sec (k,l)=(8,2)} for solving the equations mentioned above.
For each equation we work as follows. We reduce its resolution to the
problem of finding the points $(u,v)$ with integral coordinates on a 
certain elliptic curve $C$ whose equation is not in Weierstrass form. 
We find a Weierstrass model $E$ and an explicit birational transformation
\label{page birational transformation}
\begin{align*}
C\ni (u,v)&\longrightarrow(x,y)=\left(\mathcal{X}(u,v),
\mathcal{Y}(u,v)\right)\in E \\
C\ni \left(\mathcal{U}(x,y),\mathcal{V}(x,y)\right)=(u,v)
 &\longleftarrow(x,y)\in E
\end{align*}
between $C$ and $E$. This is accomplished by the {\sc maple} 
implementation of van Hoeij's algorithm  \cite{MVH et al}.
The typical point on $C$ is denoted by $P^C$ and the corresponding point
on $E$ via the above birational transformation by $P^E$. 
We will also use the notation $(u(P),v(P))$ for the coordinates 
of the point $P$ viewed as a point on $C$, hence $(u(P),v(P))=P^C$, 
and $(x(P),y(P))$ for the coordinates of the 
point $P$ viewed as a point on $E$, hence $(x(P),y(P))=P^E$.
Thus, if 
$P^C=(u,v)=(u(P),v(P))$ and $P^E=(x,y)=(x(P),y(P))$, then 
$x=\mathcal{X}(u,v)$,
$y=\mathcal{Y}(u,v)$ and $u=\mathcal{U}(x,y)$, $v=\mathcal{V}(x,y)$.

\noindent
Our problem is reduced to the following:
\begin{quote}
\vspace{-5pt}
\emph{To compute explicitly all points $P^E\in E(\Q)$ such that 
$P^C\in C(\Z)$.}
\vspace{-5pt}
\end{quote}
We deal with this problem as follows. Using 
the routine {\tt MordellWeilBasis} of 
{\sc magma}\cite{magma} 
based on the work of many contributors, like J.~Cremona, S.~Donelly, 
T.~Fisher, M.~Stoll, to mention a few of them,
we compute a 
Mordell-Weil basis for $E(\Q)$ and let $P_1^E,\ldots,P_r^E$ be generators 
of the free part of $E(\Q)$.
At this point we stress the fact that, in certain cases,
especially when the rank of the elliptic curve is $\geq 5$, it is necessary
to improve the Mordell-Weil basis computed by {\sc magma},
in the sense explained in ``Important computational issue'' of 
Appendix \ref{LLL_d_-1}; we will need to do this in 
Sections \ref{subsec d=-1} and \ref{sec (k,l)=(8,2)}.
Let $P^C=(u,v)$ denote the typical unknown point with integral 
coordinates. 
Its transformed point $P^E$ via the previously mentioned birational 
transformation is a point with rational coordinates, therefore 
$P^E=m_1P_1^E+\cdots+m_rP_r^E+T^E$, where $m_1,\ldots,m_r$ are 
unknown integers and $T^E$ denotes the typical torsion point
(only finitely many and, actually, very few options for $T^E$ exist). 
To this we associate the linear form 
\begin{equation}
                \label{eq L(P) general}
L(P)=(m_0+\frac{s}{t})\,\om_1+m_1\ellog(P_1)+\cdots+m_r\ellog(P_r)
           \;\{\pm\ellog(P_0)\}.           
\end{equation}
Some explanations have their place here. Firstly, $\ellog$ 
\label{page def frak l}
denotes the map $\mathfrak{l}: E(\R)\rightarrow\R/\Z\omega_1$ closely 
related to the elliptic-logarithm function, which is defined and 
discussed in detail in Chapter 3 of \cite{Nikosbook}, especially, 
Theorem 3.5.2.
Next, $\om_1$ is the minimal positive real period of $E$, $m_0$ is an 
extra integer whose size depends explicitly on 
$M:=\max_{1\leq i\leq r}|m_i|$, and $s,t$ are relatively prime integers 
as follows: $t\geq 1$ divides the $\mathrm{lcm}$ of the orders of the 
non-zero torsion points of $E$ and $s$ is such that $-1/2<s/t\leq 1/2$.
\footnote{Note that, by a famous theorem of B.~Mazur, $11\neq t\leq 12$;
see \cite{Mazur1}, \cite{Mazur2}, or \cite[Theorem 7.5]{Silverbook}.}
Last, the indication $\{ \}$ in the summand 
$\ellog(P_0)$ means that this 
is present only in Section \ref{sec (k,l)=(3,6)}, where $P_0$ is a certain
explicitly known point.
\\
\noindent  
The \emph{Elliptic Logarithm Method} exploits the fact that $u,v$ are
integers in order to find an upper bound for $|L(P)|$ in terms of $M$
(see \eqref{eq ub |L(P)|}) and, on the other hand, applies a deep result 
of S.~David \cite{David} in order to obtain a lower bound for $|L(P)|$ in 
terms of $M$. Comparing the two bounds of $|L(P)|$ leads to a relation
\begin{equation}
                \label{eq furnishes ub for M}
\rho M^2\leq \frac{c_{11}c_{13}}{2\th}
(\log(\alpha M+\beta)+c_{14})
(\log\log(\alpha M+\beta)+c_{15})^{r+3}+\gamma
+\frac{c_{11}}{2\th}\log\frac{c_9}{1+\th}+
\textstyle{\frac{1}{2}}c_{10},
\end{equation}
where all constants involved in it are explicit; 
see relation (9.8),
Theorem 9.1.3 of \cite{Nikosbook}.
It is clear that, if $M$ is larger than an explicit bound $B$, then 
the left-hand side is \emph{larger} than the right-hand side and this 
contradiction certainly implies that $M\leq B$. 
Since $B$ is explicit, this allows us to compute all
integer points $P^C=(u,v)$ as follows: For each $(m_1,\ldots,m_r)$ in 
the range $|m_i|\leq M$ ($i=1,\ldots,r$) we compute each point 
$P^E=m_1P_1^E+\cdots+m_rP_r^E+T^E$ with $T^E$ a torsion point and then
we compute its transformed point $P^C$ via the previously mentioned 
birational transformation. If $P^C$ has integer coordinates, then we
have gotten an integer point $P^C=(u,v)$. 

In principle, this procedure allows to pick-up all integer points $(u,v)$
and, indeed, this is so if the bound $B$ is small, say around 30. But 
the bound which we obtain from \eqref{eq furnishes ub for M} is huge
and we must reduce it to a manageable size, which is accomplished with 
de Weger's \cite{dW} technique, the basic tool of which is the 
\emph{LLL-algorithn} of Lenstra-Lenstra-Lov\'{a}sz \cite{LLL}.

\section{{\large Equation \eqref{eq general collision} with $(k,l)=(3,6)$}}
               \label{sec (k,l)=(3,6)}
Replacing in \eqref{eq general collision} $d$ by $-d$ we obtain the 
equation 
\begin{equation}
                  \label{eq d-collision}
 \binom{n}{3}=\binom{m}{6}+d\,,
\end{equation}
which we study in this section.
We have
\[
\binom{n}{3}=\dfrac{n(n-1)(n-2)}{6}=
\dfrac{1}{6}\left((n-1)^3-(n-1)\right)=\dfrac{1}{6}
\left(u^3-u\right),
\]
where $u=n-1$, and
\begin{align*}
\binom{m}{6}&=\dfrac{m(m-1)(m-2)(m-3)(m-4)(m-5)}{6!}\\
&=\dfrac{\left(\left(m-\dfrac{5}{2}\right)^2-\dfrac{25}{4}\right)
\left(\left(m-\dfrac{5}{2}\right)^2-\dfrac{9}{4}\right)\left(\left(m-
\dfrac{5}{2}\right)^2-\dfrac{1}{4}\right)}{6!}\\
&=\dfrac{\left(\dfrac{1}{2}\left(\left(m-\dfrac{5}{2}\right)^2
-\dfrac{1}{4}\right)-3\right)\left(\dfrac{1}{2}
\left(\left(m-\dfrac{5}{2}\right)^2-\dfrac{1}{4}\right)-1\right)
\dfrac{1}{2}
\left(\left(m-\dfrac{5}{2}\right)^2-\dfrac{1}{4}\right)}{6\cdot5\cdot3}
\\ &=\dfrac{(v-3)(v-1)v}{6\cdot5\cdot3},
\end{align*}
where 
$v=\dfrac{1}{2}\left(\left(m-\dfrac{5}{2}\right)^2-\dfrac{1}{4}\right)=
 (m-2)(m-3)/2$.
\\
Thus, equation \eqref{eq d-collision} implies 
\begin{equation}
                  \label{eq reduce initial}
 15(u^3-u-6d)=v^3-4v^2+3v,                 
\end{equation}
with $u,v$ related to $n$ and $m$ as above. 
We rewrite equation \eqref{eq reduce initial} as $g(u,v)=0$, where
\begin{equation}
                  \label{guv}
g(u,v)=15u^3-v^3+4v^2-90d-15u-3v.
\end{equation}
\subsection{Equation \eqref{guv} when $d=(N^3-N)/6$}
          \label{subsec d=(N^3-N)/6}
Assuming that $N$ is an explicitly known non-zero integer, we will show 
how the method of \cite[Chapter 8]{Nikosbook} can be applied in order to
compute --at least in principle-- all integer solutions of \eqref{guv}. 
A crucial fact is that certain parameters involved in the application of
that method can be expressed uniformly in $N$.  

The curve $C: g(u,v)=0$, being a non-singular cubic, has genus one.
Moreover, $(u,v)=(n,1)$ is a rational point of $C$,  
so that $C$ is a model of an elliptic curve over $\Q$. The {\sc maple} 
implementation of van Hoeij's algorithm  \cite{MVH et al} gives the 
birational transformation between $C$ and the Weierstrass model 
\begin{equation}
               \label{eq E}
  E: y^2=x^3-1575x+33750N^3-33750N-\frac{1366875}{4}N^6
+\frac{1366875}{2}N^4-\frac{1366875}{4}N^2+52650.
\end{equation}
The birational transformation from $C$ to $E$ 
mentioned in page \pageref{page birational transformation} is
\begin{align*}
(u,v)&\longrightarrow(x,y)=\left(\mathcal{X}(u,v),
\mathcal{Y}(u,v)\right)\\
\left(\mathcal{U}(x,y),\mathcal{V}(x,y)\right)=(u,v) &\longleftarrow(x,y),
\end{align*}
where the functions $\mathcal{X}$ and $\mathcal{Y}$ are 
\begin{align*}  
\mathcal{X}(u,v)= &\,(-45N^3v+45N^2uv+120N^3
-60N^2u-60Nu^2+15Nv-15uv +3v^2+60u-12v 
\\  &-60N+9) : (N-u)^2,
\\ \nonumber\mathcal{Y}(u,v)=&-\dfrac{3}{2}
(675N^6-675N^5u-675N^4u^2+675N^3u^3+120N^3v^2-120N^2uv^2-675N^4
\\ 
&+1125N^3u-480N^3v-225N^2u^2+390N^2uv-225Nu^3+90Nu^2v+420N^3
\\ \nonumber 
 &-180N^2u -180Nu^2-40Nv^2-60u^3+40uv^2+150N^2-300Nu+190Nv+150u^2
\\ & -190uv+6v^2-240N+240u-24v+18) : (N-u)^3,
\end{align*}
and the functions $\mathcal{U},\mathcal{V}$ are given by
\begin{align*}
\mathcal{U}(x,y)=\nonumber&\,(-2Nx^3-120x^2+6xy-(-15525N^3+12375N-540)x
\\&+(4050N^4-2700N^2+360N+450)y-1366875N^7+1366875N^5+135000N^4  
\\& -455625N^3-67500N^2+58725N-40500)
\\
 & \hspace{5pt} : (-2x^3-360Nx^2-(9450N^2+4050)x +2733750N^6-2733750N^4
\\& \hspace{15pt} +297000N^3+911250N^2-243000N-89100),   
\end{align*}
\begin{align*}
\mathcal{V}(x,y)=& \,
(5467500N^6-7290000N^4+759375N^3+3037500N^2-577125N-380700-
\\ & (91125N^5-60750N^3+8100N^2+10125N+8100)x-270Nx^2
\\&+(5400N^3-1800N+270)y-(-90N^2+30)xy) 
\\ & \hspace{5pt} :
(-2x^3-360Nx^2-(9450N^2+4050)x 
 +2733750N^6-2733750N^4 
\\ & \hspace{15pt} +297000N^3+911250N^2-243000N-89100).
\end{align*}
Let $\zeta$ denote the cubic root of $15$; in our computations we view
$\zeta$ as a real number.
With the aid of {\sc maple} we find out that there is exactly one 
conjugacy class of Puiseux series $v(u)$ solving $\mathrm{g}(u,v)=0$. 
This unique class contains exactly three series and only the following 
one has real coefficients:
\begin{align}\label{Puiseux_general}
v_1(u)=&\nonumber \,\zeta u+4/3+\left(\dfrac{7}{135}\zeta^2
-\dfrac{1}{3}\zeta\right)u^{-1}+\left(-\dfrac{1}{3}\zeta n^3
+\dfrac{4}{243}\zeta+\dfrac{1}{3}\zeta n\right)u^{-2}
+\left(\dfrac{7}{405}\zeta^2-\dfrac{1}{9}\zeta\right)u^{-3}
\\ &+\left(\dfrac{7}{405}\zeta^2 n^3-\dfrac{7}{405}\zeta^2 n
-\dfrac{2}{9}\zeta n^3+\dfrac{2}{9}\zeta n
+\dfrac{8}{729}\zeta-\dfrac{28}{32805}\zeta^2\right)u^{-4}+\ldots
\end{align}
In the notation of Fact 8.2.1(a) in \cite{Nikosbook}, 
$K=\mathbb{Q}(\zeta)$, $\mu_1=-1$, $\nu_1=1$ and according to  
Fact 8.2.1(d) of \cite{Nikosbook}, a constant $B_0$
can be explicitly computed with the property that, for $|u|>B_0$ the  
identity $g(u,v_1(u))=0$ holds. In our case it turns out from 
Appendix \ref{appendix B0} that we can take $B_0=|N|+1$.
Then, according to Lemma 8.3.1 in \cite{Nikosbook}, for every integer 
solution $(u,v)$ of  \eqref{guv} with $|u|\geq |N|+1$ we have $v=v_1(u)$. 
Thus in the notation of Proposition 8.3.2 in \cite{Nikosbook}, 
$\mathrm{x}(u)=\mathcal{X}(u,v_1(u))$ and, putting $u=t^{-\nu_1}=t^{-1}$ 
we write $\mathrm{x}(u)$ as a series in $t$
\begin{eqnarray}\label{xt_gen}
\lefteqn{
\mathrm{x}(t)=\nonumber45\zeta N^2-60N-15\zeta+3\zeta^2+\left(45N^3\zeta+6N
\zeta^2-120N^2-15N\zeta-4\zeta+40\right)t 
}
\\
& &\nonumber+\left(45\zeta N^4+\dfrac{34}{3}\zeta^2N^2-120N^3
-30\zeta N^2-8\zeta N-\dfrac{25}{9}\zeta^2+40N+5\zeta+3\right)t^2
\\&  &\nonumber +\left(30\zeta N^5+\dfrac{37}{3}\zeta^2 N^3-120N^4-10
\zeta N^3-\dfrac{304}{27}\zeta N^2-\dfrac{25}{9}\zeta^2N
+40N^2-\frac{44}{405}\zeta^2+6N+\frac{88}{81}\zeta \right)t^3
\\ & & +O\left(t^4\right)
\end{eqnarray}
Then the point $P_0^E$ that plays a crucial role in the resolution 
(see \cite[ Definition 8.3.3]{Nikosbook}) is 
\begin{equation}
              \label{eq P0}
P_0^E=(45N^2\zeta+3\zeta^2-60N-15\zeta,\; 90-60\zeta^2-135\zeta 
N+\dfrac{675}{2}N-\dfrac{2025}{2}N^3+180\zeta^2N^2).
\end{equation}

Now we refer to the discussion of Section \ref{sec intro} whose notation
etc we use. According to \cite[Theorem 9.1.3]{Nikosbook}, applied
to ``case of Theorem 8.7.2'', if $|u(P)|\geq\max\{B_2,B_3\}$, where 
$B_2$ and $B_3$ are explicit positive constants, then either 
$M\leq c_{12}$, where $c_{12}$ is an explicit constant, or the 
inequality \eqref{eq furnishes ub for M} holds. As already mentioned in 
Section \ref{sec intro}, all constants in \eqref{eq furnishes ub for M}
explicit. More specifically, as we show 
in Appendix \ref{appendix theta,c9,c10,c11},  
\[
B_2=3|N|,\quad B_3=|N|+1,\quad               
\th=1,\quad c_9=0.17,\quad c_{10}=\log(200|N|^3),\quad c_{11}=2,
\]
while the remaining constants appearing in 
\eqref{eq furnishes ub for M}, namely,
$\al,\be,\ga,r,\rho,c_{12},c_{13},c_{14},c_{15}$ depend
on the peculiarities of the elliptic curve $E$, like e.g. its rank and 
Mordell-Weil group which by no means can be expressed uniformly in
terms of $N$.
Thus, we have the following:
\begin{theorem}
                \label{thm furnishing ub for M}
If $|u(P)|\geq 3|N|$, then either $M\leq c_{12}$ or
\[
\rho M^2\leq c_{13}(\log(\alpha M+\beta)+c_{14})
(\log\log(\alpha M+\beta)+c_{15})^{r+3}+\gamma
+\log 0.085+
\textstyle{\frac{1}{2}}\log(200|N|^3).
\]
\end{theorem}
\subsection{Equation \eqref{guv} with $d=-1$} 
               \label{subsec d=-1}
Since $(N^3-N)/6=-1$ for $N=-2$, we can apply the general discussion of 
Section \ref{subsec d=(N^3-N)/6}, the notation of which will be used 
throughout the present section. We have
\begin{align}\label{g_d=-1}
C: g(u,v)=0, \quad \mbox{where}\quad g(u,v)=15u^3-v^3+4v^2-15u-3v+90
\end{align}
and
\begin{align}\label{fd1}
E: y^2=x^3-1575x-12451725=: f(x).
\end{align}
$E(\Q)$ has rank 5 (in the notation of 
Theorem \ref{thm furnishing ub for M} $r=5$) and trivial torsion subgroup
(in subsequent notation $r_0=1$). The free part of $E(\Q)$ is generated by the points
\[
 P_{1}^{E}=\left(235,395\right),\;
 P_{2}^{E}=\left(615,14805\right),\;
P_{3}^{E}=\left(3055,168805\right),
\]
\[
P_{4}^{E}=\left(1350, 49455\right),\;
P_{5}^{E}=\left(\frac{1185}{4},-\frac{28935}{8}\right).
\]

Actually, the Mordell-Weil basis formed by the above five points is an
improvement of the Mordell-Weil basis furnished by {\sc magma}, in the
sense of the ``Important computational issue'' of Appendix \ref{LLL_d_-1}. 

\noindent
The birational transformation between the models $C$ and $E$ is: 
%
\[
\mathcal{X}(u,v)= \frac{3(40u^2+55uv+v^2-60u+106v-277)}{(u+2)^2}
\] 
\label{page X(u,v),Y(u,v) when d=-1}
\[
\mathcal{Y}(u,v)= 
\frac{3(2505u^3+90u^2v+220uv^2+5595u^2-685uv+437v^2-6360u-1718v-15069)}
{(u+2)^3},
\]
and
\begin{eqnarray}
\mathcal{U}(x,y) & = & \frac{2x^3-60x^2+3xy-49455x+26865y+68298525}
 {-x^3+360x^2-20925x+66442950} \nonumber 
 \\ & & \label{UVd_-1} \\
\mathcal{V}(x,y) & = & \frac{15(18x^2+11xy+80325x-1311y+8004285)} 
   {-x^3+360x^2-20925x+66442950} \nonumber
\end{eqnarray}
By \eqref{Puiseux_general} and the discussion immediately after it,
for every real solution of $g(u,v)=0$ with $|u|\geq 3$ it is true that 
$v=v_1(u)$, where
\begin{align}\label{Puiseux_d_-1}
v_1(u)= &\nonumber\,\zeta u+\frac{4}{3}+\left(\dfrac{7}{135}\zeta^2
-\dfrac{1}{3}\zeta\right)u^{-1}+\dfrac{490}{243}\zeta u^{-2}
+\left(\dfrac{7}{405}\zeta^2-\dfrac{1}{9}\zeta\right)u^{-3}
\\  &+\left(-\dfrac{686}{6561}\zeta^2+\dfrac{980}{729}\zeta\right)u^{-4}+\ldots.
\end{align}
Also, by \eqref{eq P0}, 
\[
P_{0}^{E}=(3\zeta^2+165\zeta+120,\; 660\zeta^2+270\zeta+7515),
\] 
where $\zeta$ is the cubic root of $15$. 

\noindent
Referring to the discussion of Section \ref{sec intro}, we consider
the linear form 
\label{page def L(P) when d=-1}
\[
L(P)=\left(m_0+\dfrac{s}{t}\right)
\omega_1+m_1\ellog(P_1)+m_2\ellog(P_2)
+m_3\ellog(P_3)+m_4\ellog(P_4)+m_5\ellog(P_5)
\pm\ellog(P_0).
\]
Since $f(X)$ has only one real root, namely $e_1\approx234.0452973361$, 
we have $E(\R)=E_0(\R)$, therefore $\mathfrak{l}(P_i)$ coincides with the 
elliptic logarithm of $P_i^E$ for $i=1,\ldots,5$
(see Chapter 3 of \cite{Nikosbook}, especially, 
Theorem 3.5.2). 
On the other hand, $P_0^E$ has irrational coordinates.  
As {\sc magma} does not possess a routine for calculating elliptic
logarithms of non-rational points, we wrote our own routine in {\sc maple}
for computing $\ellog$-values of points with algebraic coordinates.  
\label{page my maple routine}
Thus we compute 
\[
\mathfrak{l}(P_1)\approx-0.0771021779, \quad
\mathfrak{l}(P_2)\approx-0.0404989783,\quad
\mathfrak{l}(P_3)\approx-0.0180931954,
\]
\[
\mathfrak{l}(P_4)\approx-0.0272287725,\quad
\mathfrak{l}(P_5)\approx0.0607913520, \quad
\mathfrak{l}(P_0)\approx0.1159496335.
\]
Note that the six points $P_i^E,\,i=0,1,\ldots,5$  are $\Z$-linearly 
independent because their regulator is non-zero 
(see \cite[Theorem 8.1]{Sch Zim}).
Therefore our linear form $L(P)$ falls under the scope of the 
second ``bullet'' in \cite[page 99]{Nikosbook} 
and we have $r_0=1$, $s/t=s_0/t_0=0/1=0$, $d=1$, $r=5$, $n_i=m_i$ for 
$i=1,\ldots,4$, $n_5=\pm1$, $n_0=m_0$, $k=r+1=6$, 
$\eta=1$ and $N=\max_{0\leq i\leq 5}|n_i|
\leq r_0\max\{M,\frac{1}{2}rM+1\}+\frac{1}{2}\eta r_0
=\frac{5}{2}M+\frac{3}{2}$, 
so that, in the relation (9.6) of \cite{Nikosbook} we can take 
\begin{equation}
                \label{eq d=-1, alpha,beta}
\alpha=5/2, \beta=3/2.                
\end{equation}
We compute the canonical heights of $P_1^E, P_2^E, P_3^E, P_4^E,P_5^E$ 
using {\sc magma}\footnote{For the definition of the 
canonical height we follow J.H.~Silverman; as a consequence the values 
displayed here for the canonical heights are the halves of those computed 
by {\sc magma} and the least eigenvalue $\rho$ of the height-pairing matrix 
$\mathcal{H}$ below, is half that computed by {\sc magma}; 
cf.~``Warning'' at bottom of p.~106 in \cite{Nikosbook}.
\label{foot canonical height}} 
and for the canonical height of $P_0^E$ we confine ourselves to the upper 
bound furnished by Lemma \ref{lemma hP0}. Thus we have
\[
\hat{h}(P_1^E)\approx2.2913414307, \quad 
\hat{h}(P_2^E)\approx2.0649979264, \quad 
\hat{h}(P_3^E)\approx3.3258621376,
\]
\[\hat{h}(P_4^E)\approx2.5707390271, \quad 
\hat{h}(P_5^E)\approx2.6752327982, \quad \hat{h}(P_0^E)\leq 7.300572483\,.
\]
The corresponding height-pairing matrix for the particular Mordell-Weil 
basis is
\[
\mathcal{H}=\left(\begin{array}{rrrrr}
2.2913414307 & 1.0192652309 & 1.5359254535 & -1.2315944080 & 
-0.77710896815\\
1.0192652309 & 2.0649979264 & 0.3597655203 & -0.4612024943 & 
0.3804341218\\
1.5359254535 & 0.3597655203 & 3.3258621376 & -1.9571170828 & 
-1.9878905154\\
-1.2315944080 & -0.4612024943 & -1.9571170828 & 2.5707390271 & 
1.3907956375\\
-0.7771089681 & 0.3804341218 & -1.9878905154 & .3907956375 & 2.6752327982
\end{array}\right)
\]
with minimum eigenvalue 
\begin{equation}
       \label{eq d=-1,rho}
  \rho\approx0.7722274789.       
\end{equation}
Next we apply \cite[Proposition 2.6.3]{Nikosbook} in order to compute 
a positive constant $\gamma$ 
with the property that 
$\hat{h}(P^E)-\frac{1}{2}h(x(P))\leq \gamma $ for every point 
$P^E=(x(P),y(P))\in E(\Q)$, where $h$ denotes Weil height;
\label{page def gamma}
\footnote{In the notation of \cite[Proposition 2.6.3]{Nikosbook}, 
as a curve $D$ we take the minimal model of $E$ which is $E$ itself.}
it turns out that
\begin{equation}
    \label{eq d=-1, gamma}
\gamma\approx 4. 6451703657.
\end{equation}
Finally, we have to specify the constants $c_{12},c_{13},c_{14},c_{15}$
defined in \cite[Theorem 9.1.2]{Nikosbook}. 
This is a rather straightforward task if one follows the detailed 
instructions of \cite[``Preparatory to Theorem 9.1.2'']{Nikosbook} which 
can be carried out even with a pocket calculator, except for the 
computation of various canonical heights. Clearly, this is quite
a boring job which, fortunately, can be carried out almost automatically 
with a {\sc maple} program. In this way we compute 
\begin{equation}
                \label{eq d=-1, c12,c13,c14,c15}
c_{12}\approx1.210103\cdot10^{27},\quad 
c_{13}\approx1.342820\cdot10^{281},\quad 
c_{14}\approx2.09861,\quad c_{15}\approx25.03975.                
\end{equation}
Now, in view of Theorem \ref{thm furnishing ub for M} and 
\eqref{eq d=-1, alpha,beta}, \eqref{eq d=-1,rho}, 
\eqref{eq d=-1, gamma}, \eqref{eq d=-1, c12,c13,c14,c15},
we conclude that, if $|u(P)|\geq 6$, then either $M\leq c_{12}$ or 
\begin{eqnarray*}
\lefteqn{0.77222\cdot M^2 \leq} 
\\
 & & 1.34\cdot10^{281}\cdot(\log(2.5M+1.5)
+2.0986)\cdot(\log(0.4342\log(2.5M+1.5))+25.0397)^5+5.4159.
\end{eqnarray*}
But for all $M\geq 6.3\cdot10^{147}$, we check that the left-hand side
is strictly larger than the right-hand side which implies that
$M < 6.86\cdot10^{147}$, therefore
\begin{equation}\label{BM_d_-1}
M\leq\max\{c_{12},\; 6.86\cdot10^{147}\}=6.86\cdot10^{147}
\quad\mbox{provided that $|u(P)|\geq 6$.}
\end{equation}
An easy straightforward computation shows that all integer points $P^C$ 
with $|u(P)|\leq 5$ (equivalently, all integer solutions $(u,v)$ of 
\eqref{g_d=-1} with $|u|\leq 5$) are the following:
\begin{equation}
                \label{eq d=-1 small pts}
 P^C=(-2,0),\; (-2,1),\; (-2,3),\; (-1,6),\;(0,6),\;(1,6).               
\end{equation} 
In order to find explicitly all points $P^C$ with $|u(P)|\geq 6$ it is
necessary to reduce the huge upper bound \eqref{BM_d_-1} to an upper 
bound of manageable size. This is accomplished in 
Appendix \ref{LLL_d_-1}, where we show that $M\leq 27$. 
Therefore, we have to check which points  
\[
P^E=m_1P_1^E+m_2P_2^E+m_3P_3^E+m_4P_4^E+m_5P_5^E,
   \quad\mbox{with $\max_{1\leq i\leq 5}|m_i|\leq 27,$}
\]   
have the property that $P^E=(x,y)$ maps via the transformation 
\eqref{UVd_-1} to a point $P^C=(u,v)\in C$ with integer 
coordinates. 
We remark here that every point $P^C$ with $u(P)$ integer
and $|u(P)|\geq 6$ is obtained in this way, but the converse is not 
necessarily true; i.e. if $\max_{1\leq i\leq 5}|m_i|\leq 27$ and 
the above $P^E$ maps to $P^C$ with integer coordinates, it is not
necessarily true that $|u(P)|\geq 6$.

\noindent

If we were going to check all 5-tuples $(m_1,m_2,m_3,m_4,m_5)$ in the 
range $-27\leq m_i\leq 27$ by ``brute force'' this would take more than 
$15$ days of computation. 
Therefore, we apply a simple but very effective trick to speed up 
this final search.
This trick, called in \cite{Stro Tza} \emph{inequality trick}, 
\label{page ineq trick}
is based in the observation that, for every $5$-tuple 
$(m_1,m_2,m_3,m_4,m_5)$ corresponding to a 
point $P^E=m_1P_1^E+m_2P_2^E+m_3P_3^E+m_4P_4^E+m_5P_5^E$, the upper bound
of $|L(P)|$ mentioned just above \eqref{eq furnishes ub for M}, more
specifically, 
\begin{equation}
                \label{eq ub |L(P)|}
 |L(P)| \leq k_1\exp(k_2-k_4M^2)                
\end{equation}
must be satisfied for the six-tuple $(m_0,m_1,\ldots,m_5)$ where 
$m_0$ is the extra parameter appearing in \eqref{eq L(P) general}
with $|m_0|\leq 27$.
The heuristic observation is that the above inequality is 
very unlikely to be satisfied for points $P^E$, with 
at least one large coefficient $m_i$. The reason is that the elliptic 
logarithms $\ellog(P_i)$ are more or less randomly distributed 
(at least there is no reason to assume otherwise) so that the linear 
$L(P)$ is rarely very small. 
Checking whether the $L(P)$ coming from a certain $6$-tuple 
$(m_0,m_1,m_2.m_3,m_4,m_5)$ in the range $-27\leq m_i\leq 27$ satisfies 
the above displayed inequality requires real number computations which 
are considerably faster than those required for symbolically computing   
$P^E=m_1P_1^E+m_2P_2^E+m_3P_3^E+m_4P_4^E+m_5P_5^E$ and then
checking whether the corresponding point $P^C$ is integral.
Actually, this reduces the computation to a few hours and furnishes
us with the points figuring in 
Table \ref{Table All integer points_d_-1}.

{\bf Important remark}. As mentioned in the ``Important computational
issue'' at the end of Appendix \ref{LLL_d_-1}, the online {\sc magma} 
calculator ({\tt V2.24-3})
returns a different Mordell-Weil basis for the elliptic curve \eqref{fd1}.
The value of $\rho$ corresponding to that basis is $\rho\approx 0.410937$.
As a consequence, the initial upper bound for $M$ (cf.~\eqref{BM_d_-1})
is $M<8.63\cdot 10^{147}$ and after four reduction steps, the final reduced 
upper bound is 34. Therefore the final check for all 6-tuples 
$(m_0,m_1,\ldots,m_5)$ in the range $-34\leq m_i\leq 34$ needs at
least four times ($4\approx (34/27)^6$) more computation time; actually, 
it needs much more according to our experiments.
\begin{center}
\tablecaption{All points $P^E=\Sigma_i m_iP_i^E$ with 
$P^C=(u,v)\in\Z\times\Z$.} 
                            \label{Table All integer points_d_-1}  
\tablefirsthead{ 
\multicolumn{7}{c}{} \\ \hline
& & & & & & \\[-12pt]
$m_1$ & $m_2$ & $m_3$ & $m_4$ & $m_5$ & $P^E=(x,y)$ & $P^C=(u,v)$ \\ 
\hline\hline
               }
\tablehead{ \hline
\multicolumn{7}{|l|}{{\small \sl continued from previous page}} \\ \hline
$m_1$ & $m_2$ & $m_3$ & $m_4$ & $m_5$ & $P^E=(x,y)$ & $P^C=(u,v)$
  \\
          }          
 \tabletail{\hline \multicolumn{4}{|r|}{\small\sl continued on next page}
 \\ \hline}
\tablelasttail{\hline\hline }
 \begin{mpxtabular}{|r|r|r|r|r|c|c|}
 $-1$ & $0$ & $0$ & $-1$ & $1$ & $(27075, -4455045)$ & $(-2,1)$ 
 \\ \hline
 $-1$ & $0$ & $0$ & $0$ & $0$ & $(235, -395)$ & $(1,6)$ 
 \\ \hline
  $0$ & $0$ & $-1$ & $-1$ & $0$ & $(495, -10395)$ & $(-1,6)$
 \\ \hline
$0$ & $0$ & $-1$ & $0$ & $-1$ & $(555, 12555)$ & $(-138, -339)$ 
  \\ \hline
 $0$ & $0$ & $-1$ & $0$ & $0$ & $(3055, -168805)$ & $(-2,3)$ 
 \\ \hline
$0$ & $0$ & $0$ & $0$ & $1$ & $(1185/4, -28935/8)$ & $(0,6)$ 
 \\ \hline
 \end{mpxtabular}
 \end{center}
%
Note that only the point $P^C$ which corresponds to 
$(m_1,m_2,m_3,m_4,m_5)=(0,0,-1,0,-1)$ has $|u(P)|\geq 6$. All other
points $P^C=(u,v)$, although they correspond to  
$(m_1,m_2,m_3,m_4,m_5)$ with $\max_{1\leq i\leq 5}|m_i|\leq 1$,
have $|u|<6$. These five points are of course contained in the 
already found list of points \eqref{eq d=-1 small pts}, which contains 
one more point, namely $(u,v)=(-2,0)$, because this point cannot 
correspond via the (affine) birational transformation to a point $P^E$; 
cf.~page \pageref{page X(u,v),Y(u,v) when d=-1}.
We have thus proved the following:
\begin{theorem}
               \label{thm d=-1,main}
The integer solutions of the equation \eqref{g_d=-1} are
$$(u,v)= (-138,-339),\; (-2,0),\; (-2,1),\; (-2,3),\; (-1,6),\;
(0,6),\;(1,6).$$             
\end{theorem}
\begin{corollary}
            \label{corol (k,l)=(3,6)}
No $(3,6)$ near-collision with difference 1 exists.
\end{corollary}
\begin{proof}
Assume that $(n,3,m,6)$ is a near collision with difference 1. Then 
$\binom{m}{6}-\binom{n}{3}=1$, which is equation 
\eqref{eq d-collision} with $d=-1$. 
At the beginning of Section \ref{sec (k,l)=(3,6)} we saw that if we put 
$u=n-1$ and $v=(m-2)(m-3)/2$, then
$(u,v)$ is an integer solution of the equation \eqref{eq reduce initial}
with $d=-1$, i.e. $(u,v)$ is an integral point on the curve \eqref{g_d=-1}. 
By the restrictions on the definition of collision, $n\geq 6$, 
so $u\geq 7$ and by Theorem \ref{thm d=-1,main}, no solution $(u,v)$ to 
\eqref{g_d=-1} exists with $u\geq 7$.  
\end{proof}

\subsection{Other cases with $d=(N^3-N)/6$}
               \label{othercases}

From the discussion at the beginning of Section \ref{subsec d=(N^3-N)/6}
and \eqref{guv} we will deal with the elliptic curve
$ C\,:\, 15u^3-v^3+4v^2-15u-3v-90d $.
\\ \noindent

The birationally equivalent Weierstrass model $E$ is, by \eqref{eq E}, 
$ E\,:\, y^2=x^3-1575x+a_6(N)$,
where
\[
a_6(N)= -\dfrac{1366875}{4}N^6+\dfrac{1366875}{2}N^4+33750N^3
-\dfrac{1366875}{4}N^2-33750N+52650.
\]
Generally speaking, the method for computing all integer points on $C$ is 
completely analogous to the one we applied in Section \ref{subsec d=-1}.
Moreover, for $d=1,4,10,20$ (corresponding to $N=2,3,4,5$) the final 
checking, after the reduction process (cf.~the discussion just before the 
Table \ref{Table All integer points_d_-1}) is considerably less 
time-consuming because the ranks of the elliptic curves are at most 4.
Therefore, we think it is enough to include all necessary information in
Table \ref{Table C and E}. We remind the notation which is identical to
that of Section \ref{subsec d=-1}: 
$r$ denotes rank; torsion subgroup is trivial for every $N\geq 1$, 
therefore, ``generators'' in the table means always ``generators of 
infinite rank''.
The discriminant is negative for every $N\geq 1$ and $e_1$ is the 
sole real root of the cubic polynomial in the right-hand side of the 
defining equation of $E$. Finally, $\rho$ denotes the least eigenvalue
of the (positive definite) regulator matrix.
All points $P_i,\,i=0,1,2,3,4$ below refer to the model $E$;
for simplicity in the notation we omit the superscript $E$ from them.
%
\begin{center}
\tablecaption{$C: 15u^3-v^3+4v^2-15u-3v-90d$ and 
$E: y^2=x^3-1575x+a_6(N)$} 
 \label{Table C and E}  
\tablefirsthead{ 
\multicolumn{7}{c}{} \\ \hline
& & & & & & \\[-10pt]
$N$ & $d$ & $a_6(N)$ & $r$  & Generators  &  $\rho$ & $e_1$ 
\\ \hline\hline
               }
\tablehead{ \hline
\multicolumn{7}{|l|}{{\small \sl continued from previous page}} \\ \hline
$N$ & $d$ & $a_6(N)$ & $r$  & Generators  & $\rho$ & $e_1$  
\\ \hline         }          
\tabletail{\hline \multicolumn{7}{|r|}{\small\sl continued on next page}
 \\ \hline}
\tablelasttail{\hline\hline }
 \begin{xtabular}{|c|c|c|c|l|c|c|}
  & & & & & & \\[-10pt]
  $2$ & $1$ & $-12046725$ & $2$ & 
$P_1=(26745/4, -4373685/8)$ & $1.8907445355$   
& $231.5297170832$ \\
& & & & $P_2=(2995, 163855)$ & &
 \\ \hline
  & & & & & & \\[-10pt]
$3$ & $4$ & $-195967350$ & $2$ & $P_1=(37845, 7362270)$ & $1.9685805562$ 
& $581.7501698100$ \\
& & & & $P_2=(152325, -59450670)$ & &
 \\ \hline
  & & & & & & \\[-10pt] 
$4$ & $10$ & $-1228109850$ & $3$ & $P_1=(2530,122320)$ & $2.1464178968$  
 & $1071.3824031820$ \\
& & & & $P_2=(3414, 196362)$ & & \\
& & & & $P_3=(108705/49,33758640/343)$ & &
\\ & & & & & & \\[-25pt]
\\ \hline
 & & & & & & \\[-10pt]
$5$ & $20$ & $-4916647350$ & $4$ & $P_1=(1232475, 1368255420)$ 
                                             & $1.5758474521$   
& $1700.7293293549$      \\
& & & 
& $\displaystyle{P_2=\left(2181, 73854\right)}$ & & \\
& & & 
& $\displaystyle{P_3=\left(136825, 50611330\right)}$ & & \\
& & & & $P_4=(2235, 79020)$ & & \\             
 \end{xtabular}
 \end{center}
Remark. 
In the case $N=5$, the $\rho$-value corresponding to the set of 
generators computed by the online {\sc magma} calculator 
is 0.4945449338. For reasons explained in the remark after 
Table \ref{Table many parameters}, we would 
like to have a set of generators with a $\rho$-value as large as possible. 
By applying unimodular transformations to the basis computed by 
{\sc magma} and computing the corresponding $\rho$'s we succeeded to 
compute the basis shown in Table \ref{Table C and E}.

\begin{center}
\tablecaption{The point $P_0$ (see \eqref{eq P0}; 
\mbox{$\zeta=\sqrt[3]{15}$})} 
 \label{Table P0}  
\tablefirsthead{ 
\multicolumn{4}{c}{} \\[-15pt] \hline
& & & \\[-12pt]
$N$ & $d$ & $P_0$ & Upper bound of $\hat{h}(P_0^E)$ \\ \hline\hline
               }
\tablehead{ \hline
\multicolumn{4}{|l|}{{\small \sl continued from previous page}} \\ \hline
$N$ & $d$ & $P_0$ & Upper bound of $\hat{P_0}$
  \\
          }          
 \tabletail{\hline \multicolumn{4}{|r|}{\small\sl continued on next page}\\ \hline}
\tablelasttail{\hline\hline }
 \begin{mpxtabular}{|c|c|c|c|}
  & & &  \\[-10pt]
$2$ & $1$  
&$(3\zeta^2+165\zeta-120, -660\zeta^2+270\zeta+7335)$  
& $7.6463097298$ \\ \hline
& & & \\[-10pt]
$3$ & $4$  
&$(3\zeta^2+390\zeta-180, 1560\zeta^2-405\zeta-26235)$  
& $8.539616384$ \\ \hline
& & & \\[-10pt]
$4$ & $10$ 
& $(3\zeta^2+705\zeta-240, 2820\zeta^2-540\zeta-63360)$
& $9.141125914$ \\ \hline
& & & \\[-10pt] 
$5$ & $20$ 
& $(3\zeta^2+1110\zeta-300, 4440\zeta^2-675\zeta-124785)$
& $13.29809473$ 
 \\             
 \end{mpxtabular}
 \end{center}
\noindent
Completely analogously to the case $d=-1$ in Section \ref{subsec d=-1}, 
in order to obtain an upper bound of $M$ we compute the parameters 
$\alpha,\beta, c_{12},c_{13},c_{14},c_{15}$, 
as well as the analogous to
those just above relation \eqref{eq d=-1, alpha,beta},
and apply Theorem \ref{thm furnishing ub for M}, according to which, 
either $M\leq c_{12}$ or  \eqref{eq furnishes ub for M} holds. 
Always $k=r+1$ and $\alpha,\beta$ are ``very small'' 
integers explicitly calculable. We remind that the parameters $c_{12},c_{13},c_{14},c_{15}$ are defined in 
\cite[Theorem 9.1.2]{Nikosbook} and calculated according to the
instructions in the ``Preparatory to Theorem 9.1.2'' therein.
The values of these parameters are shown in  
Table \ref{Table many parameters}.
%
\begin{center}
\tablecaption{Parameters in the computations of an upper bound for $M$} 
 \label{Table many parameters}  
\tablefirsthead{ 
\multicolumn{10}{c}{} \\[-5mm] \hline
 & & & & & & & & & \\[-10pt]
$n$ & $d$ & $c_{12}$ & $c_{13}$ & $c_{14}$ & $c_{15}$ 
& $\alpha$ & $\beta$ & $\gamma$  & $k$ 
\\ \hline\hline
              }
\tablehead{ \hline
\multicolumn{10}{|l|}{{\small \sl continued from previous page}} \\ \hline
$n$ & $d$ & $c_{12}$ & $c_{13}$ & $c_{14}$ & $c_{15}$ 
& $\alpha$ & $\beta$ & $\gamma$ & $k$  \\
          }          
 \tabletail{\hline \multicolumn{10}{|r|}
                  {\small\sl continued on next page}\\ \hline}
\tablelasttail{\hline\hline }
 \begin{mpxtabular}{|c|c|c|c|c|c|c|c|c|c|}
  & & & & & & & & & \\[-10pt]
$2$ & $1$ & $1.010\cdot10^{27}$ & $1.162\cdot10^{116}$ & $2.098$ 
& $24.097$ & $1$ & $3/2$ & $4.6396592897$ & $3$ \\ \hline
 & & & & & & & & &  \\[-10pt]
$3$ & $4$ & $2.074\cdot10^{30}$ & $1.841\cdot10^{116}$ & $2.098$ 
& $27.779$ & $1$ & $3/2$ & $5.1045188249$ & $3$ \\ \hline
 & & & & & & & & &  \\[-10pt]
$4$ & $10$ & $4.469\cdot10^{34}$ & $1.332\cdot10^{163}$ & $2.098$ 
& $31.449$ & $3/2$ & $3/2$ & $5.4103994087$ & $4$ \\ \hline
 & & & & & & & & & \\[-10pt]
 $5$ & $20$ & $8.421\cdot10^{37}$ & $3.181\cdot10^{218}$ & $2.098$ 
 & $34.224$ & $2$ & $3/2$ & $5.64159117300$ & $5$ \\       
 \end{mpxtabular}
 \end{center}
%
The upper bounds $B(M)$ of $M$ and the respective reduced upper bounds
which are obtained by a reduction process completely analogous to that of
the case $d=-1$ (Appendix \ref{LLL_d_-1}) are shown in 
Table \ref{Table B(M)}.
Finally we pick all points $P^E=\sum_im_iP_i^E$ with $|m_i|$ less
that the reduced bound, such that their corresponding
point $P^C$ has integer coordinates, as discussed at the end of
Section \ref{sec intro}. Our results are shown in 
Table \ref{Table All integer points_n}.

%
\begin{center} 
\tablecaption{Upper bounds of $M$} 
 \label{Table B(M)}  
\tablefirsthead{%
\multicolumn{4}{c}{} \\[-10pt]\hline
& &  & \\[-10pt]
$N$ & $d$ & $B(M)$: Initial bound & Reduced bound  
\\ \hline \hline
              }
\tablehead{ \hline 
\multicolumn{4}{|l|}{{\small \sl continued from previous page}} 
\\ \hline \hline
& &  & \\[-10pt]
$N$ & $d$ & $B(M)$: Initial bound & Reduced bound 
  \\  \hline
          }   
\tabletail{
\hline \multicolumn{4}{|r|}{\small\sl continued on next page}\\ \hline
          }
\tablelasttail{\hline\hline }
 \begin{xtabular}{|c|c|c|c|} 
 & & &  \\[-10pt] 
$2$ & $1$ & $4.34\cdot10^{62}$ & $14$ \\ \hline
 & & & \\[-10pt]
$3$ & $4$ & $6.74\cdot10^{62}$ & $21$ \\ \hline
  & & & \\[-10pt]
$4$ & $10$ & $1.64\cdot10^{87}$ & $26$  \\ \hline
 & & & \\[-10pt]
 $5$ & $20$ & $1.54\cdot10^{116}$  & $13$
 \\            \hline 
 \end{xtabular}
  \end{center}
%
%
\begin{center}
\tablecaption{All points $P^E=\Sigma_i m_iP_i^E$ with 
$P^C=(u,v)\in\Z\times\Z$.} 
                            \label{Table All integer points_n}  
\tablefirsthead{ 
\multicolumn{4}{c}{} \\[-15pt] \hline
& & & \\[-12pt]
$N$ & $d$ & $P^E=(x,y)$ & $P^C=(u,v)$ \\ \hline\hline
               }
\tablehead{ \hline
\multicolumn{4}{|l|}{{\small \sl continued from previous page}} \\ \hline
$N$ & $d$ & $P^E=(x,y)$ & $P^C=(u,v)$
  \\
          }          
 \tabletail{\hline \multicolumn{4}{|r|}{\small\sl continued on next page}\\ \hline}
\tablelasttail{\hline\hline }
 \begin{xtabular}{|c|c|c|c|}
$2$ & $1$ & $(2995, -163855)$ & $(2, 3)$
 \\ \hline
$3$ & $4$ & $(16855, -2188180),\; (152325, -59450670)$ & $(3, 3),\; (3,1)$
 \\ \hline
$4$ & $10$ & $(108705/49, -33758640/343),\; (55165, -12956680)$
& $(11,28),\; (4,3)$  \\
& & $(497325, -350719920)$ & $(4,1)$
\\ \hline
$5$ & $20$ & $(1232475, -1368255420),\quad (136825, -50611330)$ 
& $(5, 1),\;(5, 3)$
 \\             
 \end{xtabular}
 \end{center}

We have thus proved the following:
\begin{theorem}
               \label{thm various d,main}
For $d\in\{1,4,10,20\}$ all integer solutions of the equation
\eqref{eq reduce initial} are those listed in the fourth column of 
Table \ref{Table All integer points_n}.           
\end{theorem}
\begin{corollary}
            \label{corol (k,l)=(6,3)}          
No $(6,3)$ near-collision with difference 1 exists.            
\end{corollary}
\begin{proof}
Assume that $(n,6,m,3)$ is a near collision with difference 1. Then 
$\binom{m}{3}-\binom{n}{6}=1$ and, on interchanging $m,n$, we are led
to equation \eqref{eq d-collision} with $d=1$.
According to Section \ref{sec (k,l)=(3,6)}, if in \eqref{eq d-collision}
we put $u=n-1$ and $v=(m-2)(m-3)/2$, then
$(u,v)$ is an integer solution of the equation \eqref{eq reduce initial}
with $d=1$. Moreover, by the restrictions on the definition of collision, 
$n\geq 6$, so $u\geq 7$. 
According to Theorem \ref{thm various d,main}, for $d=1$ there is no 
solution $(u,v)$ with $u\geq 7$, and this concludes the proof.
\end{proof} 

%
%
\section{{\large Equation \eqref{eq general collision} with $(k,l)=(8,2)$}}
       \label{sec (k,l)=(8,2)}
We write our equation as follows: 
\[
\frac{\left(n^2-7n\right)\left(n^2-7n+6\right)
      \left(n^2-7n+10\right)\left(n^2-7n+12\right)}
                           {3\cdot4\cdot5\cdot6\cdot7\cdot8}
+2 = (m^2-m).
\]
Putting 
\begin{equation} 
                 \label{eq subst for 1st quartic}
u=\frac{1}{2}n^2-\frac{7}{2}n+6, \quad v=210m-105 
\end{equation}
we are led to 
\begin{equation}
                \label{eq 1st quartic}
v^2=35u^4-350u^3+945u^2-630u+315^2,                
\end{equation}
hence, it suffices to explicitly solve equation \eqref{eq 1st quartic}.
The most straightforward thing for doing this would be to turn 
to
{\sc magma}'s routine {\tt IntegralQuarticPoints} which is based on 
\cite{Tza quartic} and was firstly developed in 1999 by Emmanuel Herrmann 
and further improved in the years 2006-2013 by Stephen Donnelly and other 
people of {\sc magma} group. And indeed, we ran the above routine
for \eqref{eq 1st quartic}, but after five days, {\sc magma} gave up 
without results, with the message ``Killed''. 
Consequently we must solve \eqref{eq 1st quartic} 
``non-automatically'', following the method of 
\cite{Tza quartic}, as exposed in \cite[Chapter 6]{Nikosbook}.

For the successful accomplishment of this, crucial role play:
\begin{enumerate}
\item Our Mordell-Weil basis which is an improvement of the one 
furnished by {\sc magma}, as explained in the ``Important remark'' at 
the end of this section, and\\
\item  The application of an \emph{inequality trick} completely 
analogous to that which we discuss a little before and after relation 
\eqref{eq ub |L(P)|}.
\end{enumerate}

\subsection{The equation $v^2=35u^4-350u^3+945u^2-630u+315^2$}
               \label{subsec first_Equat}
We will deal with the elliptic curve
\[
C: v^2=Q(u):=35u^4-350u^3+945u^2-630u+315^2.
\]
We use the notation, results etc of  \cite[Chapter 6]{Nikosbook};
thus we have $a=35$, $b=-350$, $c=945$, $d=-630$, $e=315$.
By \cite[Relation (6.3)]{Nikosbook} the Weierstrass model which is 
birationally equivalent to the curve $C$ is
\begin{equation}
         \label{eq E for quartic}
E: y^2=f(x):=x^3+Ax+B,         
\end{equation}
where $A=-13968675$ and $B=3410363250$, and the birational functions
\[
C\ni (u,v) \mapsto (\mathcal{X}(u,v)\,,\mathcal{Y}(u,v))=(x,y)\in E 
\]
\[
E\ni (x,y) \mapsto (\mathcal{U}(x,y)\,,\mathcal{V}(x,y))=(u,v)\in C 
\]
are 
\begin{eqnarray}\label{XY1}
\nonumber\mathcal{X}(u,v) & = & \,\frac{315(u^2-2u-2v+630)}{u^2}
\\[-3pt] 
\\[-3pt]  \nonumber
\mathcal{Y}(u,v) & = & \,
-\frac{630(175u^3-945u^2-uv+945u+630v-198450)}{u^3}\,,
\end{eqnarray}
(\cite[Relation (6.4)]{Nikosbook}), and
\begin{eqnarray}
\mathcal{U}(x,y) & = & -\frac{630(x+109935+y)}{x^2-630x-13792275}\nonumber
\\[-3pt] & & \label{UV1}
\\[-3pt] 
\mathcal{V}(x,y) & = &
\nonumber-315(x^4+630x^3+2x^2y-529200x^2+439740xy+22441718250x
\\  & &\hspace{35pt}-110933550y-196956864680625)
                             :(x^2-630x-13792275)^2 \nonumber
\end{eqnarray}
(\cite[Relations (6.5), (6.6)]{Nikosbook}).
\\
The roots $e_1>e_2>e_3$ of $f(x)$ have approximate values
\[
e_1\approx3608.8322706141>e_2\approx245.1990070867
>e_3\approx-3854.0312777009.
\]
A fundamental pair of periods for the Weierstrass $\wp$ function 
associated to $E$ is
\[ 
\omega_1\approx0.043947022525096,\quad
\omega_2\approx0.042006613806929\cdot i.
\]
Now we refer to Section \ref{sec intro} the notation etc of which 
we adopt here.
\\
The rank of $E$ is $5$ and the torsion subgroup 
$E_{tors}(\mathbb{Q})$ is trivial. 
The following points form a Mordell-Weil basis for $E(\mathbb{Q})$:
\footnote{See the ``Imporatnt remark at the end of this 
section.\label{foot important quartic}}
$$P_1^E=(-1799,150724),\; P_2^E=(105,-44100),\; P_3^E=(-315,-88200),$$ 
$$P_4^E=(8985,776700),\; P_5^E=(3885,88200).$$
We note that, for $i=1,2,3$, the points $P_i^E$ belong to 
$E_1(\mathbb{R})$, the bounded piece (``egg'') of $E(\R)$, therefore  
by ``Conclusions and remarks'' (1) in \cite[page 51]{Nikosbook}, 
$\mathfrak{l}(Pi)$ is the elliptic logarithm of the point $P_i^E+Q_2^E$, 
where $Q_2^E=(e_2,0)$. Now $P_i^E+Q_2^E$ belongs to the infinite piece 
$E_0(\R)$ of $E(\R)$ but its coordinates are non-rational, belonging to the 
cubic extension of $\Q(e_2)/\Q$, therefore, for $i=1,2,3$ we compute the 
elliptic logarithm of $P_i^E+Q_2^E$ using our {\sc maple} routine
(cf.~page \pageref{page my maple routine}); thus we find
\[
\ell_1:=\mathfrak{l}(P_1)\approx-0.1233994082363,\; 
\ell_2:=\mathfrak{l}(P_2)\approx0.318524714651,\;
\ell_3:=\mathfrak{l}(P_3)\approx0.635691508151.
\]
The points $P_4^E$ and $P_5^E$ belong to $E_0(\mathbb{R})$, therefore their 
$\mathfrak{l}$-values are equal to their respective elliptic logarithms;
thus we find
\[
 \ell_4:=\mathfrak{l}(P_4)\approx-0.1074268089,\;
 \ell_5:=\mathfrak{l}(P_5)\approx-0.18720073188.
\]
Next we need to calculate approximate values of the canonical 
heights:\footnote{See footnote \ref{foot canonical height}.}
\[
\hat{h}(P_1^E)\approx2.7309763445, \; \hat{h}(P_2^E)\approx1.2722439353,  
\; \hat{h}(P_3^E)\approx1.0972517248,
\] 
\[
\hat{h}(P_4^E) \approx2.5539836387, \; 
\hat{h}(P_5^E)\approx1.2394130665
\]
and the height-pairing matrix
\[
\mathcal{H}=\left(\begin{array}{rrrrr}
2.2913414307 & 1.0192652309 & 1.5359254535 & -1.2315944080 & -0.77710896815
\\
1.0192652309 & 2.0649979264 & 0.3597655203 & -0.4612024943 & 0.38043412180\\
1.5359254535 & 0.3597655203 & 3.3258621376 & -1.9571170828 & -1.98789051540
\\
-1.2315944080 & -0.4612024943 & -1.9571170828 & 2.5707390271 & 1.39079563750
\\
-0.7771089681 & 0.3804341218 & -1.9878905154 & 0.3907956375 & 2.67523279820
\end{array}\right)
\]
with minimum eigenvalue \footref{foot important quartic}
$$\rho\approx0.5764009469.$$
We will need also to compute a positive number $\gamma$ such that 
$\hat{h}(P^E)-\frac{1}{2}h(x(P))\leq \gamma $, where $h$ denotes Weil height.
This we do by applying Proposition 2.6.3 of \cite{Nikosbook}. 
In the notation of that proposition, as a curve $D$ we take the minimal 
model of $E$ which is $E$ itself and, following the simple instructions 
therein we compute 
$\gamma=6.4974558131.$
Finally, in order to compute the necessary constants involved in
\cite[Theorem 9.1.2]{Nikosbook} which are necessary for the application of
\cite[Theorem 9.1.3]{Nikosbook}, we replace the pair of fundamental 
periods $\omega_1,\omega_2$ for which $\tau:=\omega_1/\omega_2$ does not 
belong to the fundamental region of the complex upper half-plane, by the 
pair 
$(\varpi_1,\varpi_2)=(\omega_2,-\omega_1)$; for this pair, 
$\tilde{\tau}:=\varpi_1/\varpi_2$ satisfies $|\tilde{\tau}|\geq 1$, 
$\Im\tilde{\tau}>0$ and $|\Re\tilde{\tau}|<1/2$, hence belongs to the 
fundamental region. 

In order to obtain a relation of the form \ref{eq furnishes ub for M}
we will apply Theorem 9.1.3 `` Case of Theorem 6.8'' of \cite{Nikosbook}. 
That theorem is applicable for points $P^C=(u(P),v(P))$ 
for which $|u(P)|$ is sufficiently large. 
Table 6.1 in \cite[Chapter 6]{Nikosbook} indicates a procedure for
computing how large $|u(P)|$ should be; actually, we must have
$|u(P)|\geq \max\{u^{**},\overline{u}^{**}\}$ and 
$u^{**},\overline{u}^{**}$ 
are calculated as explained in that table.
The existence of two columns in Table 6.1 of \cite[Chapter 6]{Nikosbook} 
and its specialization to our case which is 
Table \ref{Table quartic parameters} below, is explained as follows:
At this stage it is convenient, instead of searching for solutions 
of $Q(u)=v^2$ with $v\geq 0$ and $u$ of whatever sign, to look for 
solutions of both equations $Q(u)=v^2$ and $\bar{Q}(u):=Q(-u)=v^2$ with
$u,v\geq 0$. Thus, a ``bar'' over a constant refers to the second equation.
\\ \noindent
The constant $\max\{c_7,\bar{c_7}\}$($=13$ in our case) 
is used in the application of Theorem 9.1.3 `` Case of Theorem 6.8'' 
of \cite{Nikosbook}. 
%
\begin{center}
\tablecaption{Parameters and auxilliary functions for the solution of the 
quartic elliptic equation according to the table 6.1 in \cite{Nikosbook}} 
 \label{Table quartic parameters}  
\tablefirsthead{ 
\multicolumn{2}{c}{} \\[-3mm] \hline
 &  \\[-10pt]
$Q(u)=35u^4-350u^3+945u^2-630u+99225$ & 
$\overline{Q}(u)=35u^4+350u^3+945u^2+630u+99225$ \\ \hline\hline
              }
\tablehead{ \hline
\multicolumn{2}{|l|}{{\small \sl continued from previous page}} \\ \hline
 &  \\[-10pt]
$Q(u)=35u^4-350u^3+945u^2-630u+99225$ & 
$\overline{Q}(u)=35u^4+350u^3+945u^2+630u+99225$   \\ \hline\hline
          }          
\tabletail{\hline 
\multicolumn{2}{|r|}{{\small \sl continued on next page}} \\  \hline
          }
\tablelasttail{\hline\hline}
 \begin{xtabular}{|c|c|}
  & \\[-10pt]
$\sigma=1$ & $\overline{\sigma}=-1$  \\ \hline
 &   \\[-10pt]
 &  \\[-10pt]
$\mathrm{x}(u)=\dfrac{315\left(u^2-2u+630+2(Q(u))^{1/2}\right)}{u^2}$ 
& 
$\bar{\mathrm{x}}(u)=
\dfrac{315\left(u^2+2u+630+2(\overline{Q}(u))^{1/2}\right)}{u^2}$                        \\[-10pt] &  \\ \hline
\\[-25pt] &  \\
$u^{**}=3$ and $c_7=13$ & $\overline{u}^{**}=80$ and $\bar{c_7}=13$  
\\ \hline
 &  \\[-10pt]
 $P_0^E=(630\sqrt{35}+315,110250+630\sqrt{35})$ & 
$\overline{P} _0^E=
    (630\sqrt{35}+315,-110250-630\sqrt{35})$ \\ \hline
 &  \\[-10pt]
  $\mathfrak{l}(P_0)$ & $\mathfrak{l}(\overline{P}_0)=-\mathfrak{l}(P_0)$ 
  \\ \hline
 &  \\[-10pt]
 $L(P)=\mathfrak{l}(P)-\mathfrak{l}(P_0)$ & $\overline{L}(P)
 =\mathfrak{l}(P)+\mathfrak{l}(P_0)$ \\       
 \end{xtabular}
 \end{center}

\noindent
From Table \ref{Table quartic parameters} it follows that the conditions
of \cite[Theorem 6.8]{Nikosbook} which are necessary also for the 
application of \cite[Theorem 9.1.3]{Nikosbook} are fulfilled for all points 
$P^C\in C(\Z)$ with $v(P)>0$ and $|u(P)|\geq 80$. A quick computer search 
shows that the only points in $P^C(\Z)$ with $|u(P)|<80$ are those points 
$(u,v)$ listed in Table \ref{Table All integer points_d_-1} with $|u|<80$.  
\\ \noindent
From Table \ref{Table quartic parameters} it follows that, on applying
Theorem 9.1.3 of \cite{Nikosbook} we must take
$c_7=13$ and $L(P)=\mathfrak{l}(P)\pm\mathfrak{l}(P_0)$.
We have already computed approximations of the coefficients
$\omega_1$ and $\ell_i$ ($i=1,\ldots,5$) of the linear 
form $\mathfrak{l}(P)$, and using our {\sc maple} routine mentioned in 
page \pageref{page my maple routine} we also compute 
$\ell_0:=\mathfrak{l}(P_0)\approx -0.179410143$.

Using the routine {\tt IsLinearlyIndependent} of {\sc magma}, 
we see that the points $P_i^E$ ($i=0,\ldots,5$) are $\mathbb{Z}$-linearly 
independent, so that we are in the situation described in the second 
``bullet'', page 99 in \cite{Nikosbook}. Therefore, the parameters 
in the linear form (9.2) of \cite{Nikosbook} are
\[
k=r+1=6,\; d=1,\; r_0=1,\; 
(n_1,n_2,n_3,n_4,n_5)=(m_1,m_2,m_3,m_4,m_5), 
\; n_6=\pm1,\; \ell_6=\ell_0.
\]
In the notation of \cite[relation (9.3)]{Nikosbook} we have 
$N_0=\dfrac{5}{2}M+\dfrac{3}{2}$, hence $(\alpha,\beta)=(5/2,3/2)$.

In order to compute various constants involved in the upper bound for $M$
furnished by Theorem 9.1.3 of \cite{Nikosbook}, we also need to 
compute $\hat{h}(P_0^E).$ Since $P_0^E$ is not a rational point we confine 
ourselves to a reasonably good upper bound of its canonical height which 
we obtain from Proposition 2.6.4 in \cite{Nikosbook}. In the notation of 
that proposition we take as curve $D$ our curve $E$ and obtain the bound
$\hat{h}(P_0^E)\leq14.72$.

We see that the degree of the number field generated by the coordinates of 
all points $P_i$ ($i=0,\ldots,5$) is 6, so that $D=6$ in the notation of 
``Preparatory to Theorem 9.1.2'' of \cite{Nikosbook}. Following the 
instructions in that ``Preparatory'' and Theorem 9.1.2 we compute 
\[
c_{12}=6.7621175190\cdot10^{30},\quad c_{13}=3.6856632904\cdot10^{286},
\]
\[
c_{14}=2.7917594692,\quad c_{15}=28.9071122373
\]
and in the notation of \cite[Theorem 9.1.3]{Nikosbook},
\[
c_{16}=0.6761234039,\quad c_{17}=1.831780823   ,\quad c_{18}=1.
\]
By that theorem, which in our case is 
Theorem \ref{thm furnishing ub for M},
we conclude: either $M\leq c_{12}$, or $\mathcal{B}(M)>0$, where
$ \mathcal{B}(M)=
c_{18}c_{13}(\log(\alpha M+\beta)+c_{14})
(\log\log(\alpha M+\beta)+c_{15})^{k+2}+\gamma
+c_{18}\log c_{16}+ c_{17} -\rho\cdot M^2.
$
Note that all parameters of $\mathcal{B}(M)$ have already been computed and 
are displayed in this and the previous pages. 
Now it is straightforward to check that for $M\geq6.28\cdot10^{150}$ we 
have $\mathcal{B}(M)<0$, which implies that 
\[
M\leq\max\{c_{12},6.28\cdot10^{150}\}=6.28\cdot10^{150}.
\]
We cannot obtain an upper bound for $M$ essentially better than the above 
using \cite[Theorem 9.1.3]{Nikosbook}; indeed, we check that 
$\mathcal{B}(6.27\cdot10^{150})>0$ which shows that a ``little smaller'' 
bound for $M$ does not lead to a contradiction. 
\\ \noindent
We are now in a situation completely similar to that after 
relation \eqref{BM_d_-1}. There, we reduced the huge upper bound of $M$ 
by working as explained in Appendix \ref{LLL_d_-1}. 
Here, we work similarly to obtain a small upper bound for $M$. 
This time the reduction process is repeated three times to successively
give the upper bounds 170, 30 and 28 for $M$; 
the last upper bound cannot be further reduced. 
Next, we check which points $P^E=m_1P_1^E+\cdots +m_5P_5^E$ in the range 
$\max_{1\leq i\leq m}|m_i|\leq 28$ correspond to a point $P^C$ with
integral coordinates, using the \emph{inequality trick}, as explained
in the last paragraph above Table \ref{Table All integer points_d_-1}. 
The computation on a computer Intel i5-7200U @ 2.50GHz took a little more 
than 70 hours of computation and the results are comprised in 
Table \ref{Table All integer points quartic}.
\begin{theorem}
                \label{thm sols of quartic eq}
All integer solutions of the equation \eqref{eq 1st quartic} are those
listed in the seventh column of 
Table \ref{Table All integer points quartic}.                
\end{theorem}
%
\begin{center}
\tablecaption{All points $P^E=\Sigma_i m_iP_i^E$ with 
$P^C=(u,v)\in\Z\times\Z$.} 
                            \label{Table All integer points quartic}  
\tablefirsthead{ 
\multicolumn{7}{c}{} \\[-20pt] \hline 
& & & & & &\\[-12pt]
$m_1$ & $m_2$ & $m_3$ & $m_4$ & $m_5$ & $P^E=(x,y)$ & $P^C=(u,v)$ 
  \\ \hline\hline
               }
\tablehead{ \hline
\multicolumn{7}{|l|}{{\small \sl continued from previous page}} 
\\ \hline 
& & & & & &\\[-12pt]
$m_1$ & $m_2$ & $m_3$ & $m_4$ & $m_5$ & $P^E=(x,y)$ & $P^C=(u,v)$
\\ \hline\hline
           }          
\tabletail{\hline 
\multicolumn{7}{|r|}{{\small\sl continued on next page}} \\ \hline
          }
\tablelasttail{\hline\hline }
 \begin{mpxtabular}{|r|r|r|r|r|c|c|}
 $0$ & $0$ & $0$ & $0$ & $1$ & $(3885, 88200)$ & $(111, -69615)$ 
 \\ \hline
 $1$ & $1$ & $1$ & $1$ & $-1$ & $(-4427535/1369, 6153669900/50653)$ & $(111, 69615)$ 
 \\ \hline
  $0$ & $0$ & $0$ & $1$ & $-1$ & $(5355, 286650)$ & $(-22, -3535)$ 
 \\ \hline
  $1$ & $1$ & $1$ & $0$ & $1$ & $(-465570/121, 18522000/1331)$
      & $(-22, 3535)$ 
 \\ \hline
  $0$ & $0$ & $1$ & $0$ & $-1$ & $(-3570, 88200)$ & $(-102, 64575)$ 
 \\ \hline
  $1$ & $1$ & $0$ & $1$ & $1$ & $(1228395/289, 709061850/4913)$ 
            & $(-102, -64575)$ 
 \\ \hline
  $0$ & $0$ & $1$ & $0$ & $0$ & $(-315, -88200)$ & $(1, 315)$ 
 \\ \hline
  $1$ & $1$ & $0$ & $1$ & $0$ & $(396585, -249738300)$ & $(1, -315)$ 
 \\ \hline
  $0$ & $1$ & $-1$ & $1$ & $-1$ & $(4110, 124200)$ & $(-294, -520065)$ 
 \\ \hline
  $1$ & $0$ & $2$ & $0$ & $1$ & $(-170085/49, 34428150/343)$ 
      & $(-294, 520065)$ 
 \\ \hline
  $0$ & $1$ & $0$ & $0$ & $0$ & $(105, -44100)$ & $(3, 315)$ 
 \\ \hline
  $1$ & $0$ & $1$ & $1$ & $0$ & $(44205, -9261000)$ & $(3, -315)$ 
 \\ \hline
  $0$ & $1$ & $0$ & $0$ & $1$ & $(-2765, 144550)$ & $(36, 6615)$ 
 \\ \hline
  $1$ & $0$ & $1$ & $1$ & $-1$ & $(14665/4, 307475/8)$ & $(36, -6615)$ 
 \\ \hline
  $0$ & $1$ & $0$ & $1$ & $-1$ & $(-1491, 144648)$ & $(15, 945)$ 
 \\ \hline
  $1$ & $0$ & $1$ & $0$ & $1$ & $(3801, -72324)$ & $(15, -945)$ 
 \\ \hline
  $0$ & $1$ & $0$ & $1$ & $0$ & $(-9135/4, -1223775/8)$ & $(-4, 385)$ 
 \\ \hline
  $1$ & $0$ & $1$ & $0$ & $0$ & $(28035, 4652550)$ & $(-4, -385)$ 
 \\ \hline
  $0$ & $1$ & $1$ & $0$ & $-1$ & $(4761, 211716)$ & $(-35, -8295)$ 
 \\ \hline
  $1$ & $0$ & $0$ & $1$ & $1$ & $(-3771, 49608)$ & $(-35, 8295)$ 
 \\ \hline
  $0$ & $1$ & $1$ & $0$ & $0$ & $(11235, -1124550)$ & $(6, -315)$ 
 \\ \hline
  $1$ & $0$ & $0$ & $1$ & $0$ & $(210, 22050)$ & $(6, 315)$ 
 \\ \hline
  $0$ & $0$ & $1$ & $0$ & $1$ & $(12105, 1268100)$ & $(-7, -595)$ 
 \\ \hline
  $1$ & $0$ & $0$ & $1$ & $-1$ & $(-3195, -124200)$ & $(-7, 595)$ 
 \\ \hline
  $1$ & $1$ & $1$ & $1$ & $0$ & $(-629, -109306)$ & $(0, 315)$ 
 \\  \hline
  $0$ & $0$ & $0$ & $0$ & $0$ & $\mathcal{O}$ & $(0,-315)$ 
 \\ \hline
 \end{mpxtabular}
\end{center}
{\bf Important remark}. The online {\sc magma} calculator ({\tt V2.24-3})
returns the following Mordell-Weil basis for the elliptic curve 
\eqref{eq E for quartic}:
\[
(19705/81,3758300/729),\;
(14665/4,-307475/8),\;
(8985,-776700),\]
\[
(693805,-577896200),\;
(28035,-4652550).
\]
\\ \noindent
The value of $\rho$ corresponding to that basis is $\rho\approx 0.1284705$.
As a consequence, the initial upper bound for $M$ is 
$M<1.34\cdot 10^{151}$.
This not essentially better than the above 
displayed upper bound for $M$. \emph{However} after four reduction steps
--and here $\rho$ plays its important role-- 
the final reduced upper bound is 62 which cannot be further improved. 
Therefore, had we used the above Mordell-Weil basis, 
the final check for all 6-tuples $(m_0,m_1,\ldots,m_5)$ in the range 
$-62\leq m_i\leq 62$ would be at least $(62/28)^6$ times 
more expensive, which amounts to \emph{at least one year of 
computation time!}
\noindent
We must also check the points $(x,y)\in E(\mathbb{Q})$ which are zeros of 
$q(x)=x^2-630x-13792275$ appearing in the denominator of  
$\mathcal{U}(x,y)$ and $\mathcal{V}(x,y)$. But the zeros of 
$q(x)$ are irrational, so we do not have any new solutions.
 
Finally we come back to the collision equation $\binom{m}{2}=\binom{n}{8}+1$ 
from which we started. We have  
$m=(v+105)/210$, hence $105|v$, and $2u=n^2-7n+12$. 
The only solutions $(u,v)$ with $v$ divisible by $105$ are those listed in 
Table \ref{Table 1st collision}, where also the corresponding values
of $(m,n)\in\N^2$ are listed.
%
%
\begin{center}
\tablecaption{Positive integer solutions of the collision equation 
$\binom{m}{2}=\binom{n}{8}+1$} 
 \label{Table 1st collision}  
\tablefirsthead{ 
\multicolumn{2}{c}{} \\[-15pt] \hline
 &  \\[-10pt]
$(u,v)$ & $(m,n)\in\mathbb{N}^2$ \\ \hline\hline
              }
\tablehead{ \hline
\multicolumn{2}{|l|}{{\small \sl continued from previous page}} \\ \hline
&  \\[-10pt]
$(u,v)$ & $(m,n)\in\mathbb{N}^2$   \\  \hline \hline
           }
\tabletail{\hline 
\multicolumn{2}{|r|}{{\small\sl continued on next page}}\\ \hline
     }
\tablelasttail{\hline}
 \begin{xtabular}{|c|c|} \hline
$(1,315)$ & $(2,5)$, $(2,2)$  \\ \hline
$(3,315)$ & $(2,6)$, $(2,1)$  \\ \hline
$(36,6615)$ & $(32,12)$ \\ \hline
 $(15,945)$ & $(5,9)$ \\ \hline
  $(6,315)$ & $(2,0)$, $(2,7)$ \\ \hline    
    $(0,315)$ & $(2,4)$, $(2,3)$ \\ \hline    
 \end{xtabular}
\end{center}
Note that no pair $(m,n)$ in the above table satisfies the 
condition $m\geq 4$ and $n\geq 16$, therefore we have proved the following:
\begin{corollary} \label{corol (k,l)=(2,8)} 
There is no $(8,2)$ near-collision with difference 1.
\end{corollary}

\begin{appendices}
\section{{\large The constant $B_0$ in Sections \ref{subsec d=(N^3-N)/6} and \ref{subsec d=-1}}}
        \label{appendix B0}
In \cite[Fact 8.2.1\,(d)]{Nikosbook} $B_0$ denotes 
the maximum modulus of the roots of the polynomial
$\Res_{\upsilon}(\g,\dfrac{\partial\g}{\partial v})\in\Z[u]$.
Since the property that we actually need is the convergence of a certain 
power series in $u$ for $|u|>B_0$, we can take as $B_0$ any number
larger than this maximum modulus.
\begin{lemma} 
                 \label{lemma B0}
The maximum modulus of the roots of the polynomial
$\Res_{\upsilon}(\g,\dfrac{\partial\g}{\partial v})\in\Z[u]$ is 
$<|n|+1$. Therefore we can take $B_0=|N|+1$.
\end{lemma}
\begin{proof} 
(Based on an idea of E.~Katsoprinakis, whom we thank.)
We have  
\[
\Res_{\upsilon}(\mathrm{g},
\dfrac{\partial\mathrm{g}}{\partial\upsilon})
=u^6-2u^4+a_1u^3+u^2-a_1u+a_0,
\]
where $a_1=-2N^3+2N+\dfrac{8}{81}$ and 
$a_0=N^6-2N^4-\dfrac{8}{81}N^3+N^2+\dfrac{8}{81}N-\dfrac{4}{675}$.

$\Res_{\upsilon}(\mathrm{g},
\dfrac{\partial\mathrm{g}}{\partial\upsilon})
=0\Leftrightarrow(u^3-u)^2+a_1(u^3-u)+a_0=0$. 
If $u^3-u=y$ then we solve the quadratic equation $y^2+a_1y+a_0=0$ 
and we find 
$y_1=N^3-N+\dfrac{4}{81}+\dfrac{14\sqrt{7}}{405}$ and 
$y_2=N^3-N-\dfrac{4}{81}\dfrac{14\sqrt{7}}{405}$.

We have to solve the cubic equation $u^3+p u+q=0$, with $p=-1$ and 
$q=y_1$ or $y_2$.
We find that $\dfrac{q^2}{4}+\dfrac{p^3}{27}>0$ (for $|n|\geq2$), so the 
cubic equation  
$u^3+p u+q=0$ has one real root and two conjugates complex roots. So from 
Cardano's method we have that the roots are
\begin{align*}
&u_1=\A+\B\\
&u_2=-\dfrac{1}{2}(\A+\B)+i\dfrac{\sqrt{3}}{2}(\A-\B)\\
&u_2=-\dfrac{1}{2}(\A+\B)-i\dfrac{\sqrt{3}}{2}(\A-\B)
\end{align*}
where $\A=\sqrt[3]{-\dfrac{q}{2}+\sqrt{\dfrac{q^2}{4}+\dfrac{p^3}{27}}}$ 
and
$\B=\sqrt[3]{-\dfrac{q}{2}-\sqrt{\dfrac{q^2}{4}+\dfrac{p^3}{27}}}$. 

We have that $\A\cdot\B=\dfrac{1}{3}$, so 
\begin{equation}\label{cardano}
|u_2|^2=u_1^2-1
\end{equation}
So the polynomial 
$\Res_{\upsilon}(\mathrm{g},\dfrac{\partial\mathrm{g}}{\partial\upsilon})$ 
has two real roots (the real roots are in the interval $(-|N|-1,|N|+1)$)
and two pair of conjugates complex roots. From \eqref{cardano} we have 
that, if $\rho$ is root of polynomial then $|\rho|<|N|+1$.
\end{proof}
\section{{\large The constants $\th,c_9,c_{10},c_{11}$ in Section \ref{sec (k,l)=(3,6)}} } 
        \label{appendix theta,c9,c10,c11}
In order to compute the constants $\th,c_9,c_{10},c_{11}$ which are 
necessary for the resolution of equation \eqref{guv} (see the paragraph
before Theorem \ref{thm furnishing ub for M}), we follow the 
detailed instructions of \cite[Chapter 8]{Nikosbook}, especially 
sections 8.5 and 8.6 therein.
One needs first compute three positive constants $B_1,B_2$ and $B_3$
with the property that Theorem \ref{thm furnishing ub for M}) holds 
for all points $P$ with $|u(P)|\geq \max\{B_2,B_3\}$, and the computation 
of $B_2,B_3$ requires the computation of the positive constant $B_1$ with 
the following property (cf.~\cite[Proposition 8.3.2]{Nikosbook}): 
$B_1\geq B_0$ (for $B_0$ we refer to Appendix \ref{appendix B0}), 
the functions $\bx(u):=\mathcal{X}(u,v_1(u))$ and 
$\by(u):=\mathcal{Y}(u,v_1(u))$ are strictly monotonous in the interval 
$(B_1,+\infty)$ and $\by$ does not change sign in this interval.
Therefore the values of the parameters that figure in the title of this 
appendix and are involved in \eqref{eq furnishes ub for M} are computed 
under this restriction on $u(P)$.  

The following lemma is used in the computation of $B_1$; it is the correct
version of Lemma 8.5.1 in \cite{Nikosbook}.
\begin{lemma} 
                \label{lemma 8.5.1}
Let $F\in\R[X,Y]$ be a polynomial such that $F(X,0)\neq 0$ and let 
$V:\R\rightarrow\R$ be a continuous function, such that $F(u,V(u))=0$ for 
$|u|>U_0$, where $U_0$ is a positive constant.
Let $\mathcal{R}$ be the set of all \emph{real} roots of the polynomial 
$F(X,0)$, and define
$U_{min}=\min\{-U_0,\min\mathcal{R}\}$ and 
$U_{max}=\max\{U_0,\max\mathcal{R}\}$.  
Then the function $V$ keeps a constant sign in the interval 
$(U_{max}\,,\,+\infty)$
and so it does in the interval $(-\infty\,,\,U_{min})$.
\end{lemma}
\begin{proof}
Contrary to the hypothesis, assume, for example, that $V$ changes sign
in the interval $(U_{max}\,,\,+\infty)$. Then, by the coninuity of $V$,
it follows that there exists a root, say $u_0$, of $V$ and $u_0>U_{max}$.
Since $U_{max}\geq U_0$, we have $|u_0|=u_0 > U_{max}\geq U_0$, 
therefore, by hypothesis, $F(u_0,V(u_0))=0$, hence $F(u_0,0)=0$. This 
means that $u_0\in\mathcal{R}$, therefore $u_0\leq\max\mathcal{R}$. But, 
on the other hand, $u_0>U_{max}\geq \max\mathcal{R}$ and we arrive at a 
contradiction.

Similarly we arrive at a contradiction if we assume that $V$ changes sign
in the interval $(-\infty\,,\,U_{min})$.
\end{proof}

For the computation of $B_1$ we will apply Lemma \ref{lemma 8.5.1}. 
Based on this lemma and following the detailed instructions and
notation of \cite[Section 8.5]{Nikosbook} we compute a number of 
constants, namely, $R$, $M_{1,max}$,$M_{1,min}$ and $M_{2,max}$,
$M_{2,min}$. This task requires several computational steps, which we
perform with the aid of {\sc maple}. Below we give just a rough
description of the kind of computations that we have to do; the notation
is that of \cite[Section 8.5]{Nikosbook}.
\\
First, it is easy to compute that we can take $R=|N|$.
For the computation of $M_{1,max}$,$M_{1,min}$, we need compute the 
polynomial $H_1$ in the variables $u,v,Y$ satisfying 
$H_1(u,v_1(u),\mathcal{Y}(u,v_1(u))=0$. It is
the sum of $16$ monomials and $\deg_uH_1=3$, $\deg_vH_1=2$, 
$\deg_YH_1=1$.
Since $|u|>B_0=N+1$, we have $g(u,v_1(u))=0$. We also have 
$H_1(u,v_1(u),\mathcal{Y}(u,v_1(u))=0$. This leads us to consider
the resultant $R_1$ with respect of the variable $v$ of the polynomials
$g(u,v)$ and $H_1(u,v,Y)$, so that $R_1(u,\mathcal{Y}(u,v))=0$.
As it turns out, $R_1(u,Y)$ is the product of $-(-u+N)^6$ with a 
polynomial in $u$ and $Y$. Since $-u+N\neq 0$ for $|u|>B_0=N+1$,
we must have $R_{10}(u,\mathcal{Y}(u,v_1(u)))=0$. The polynomial
$R_{10}(u,Y)$ is the sum of $94$ monomials and 
$\deg_uR_{10}=3$, $\deg_YR_{10}=3$. 
We apply Lemma \ref{lemma 8.5.1}  with $F=R_{10}$, $U_0=B_0=|N|+1$, 
$V=Y$. The polynomial $R_{10}(u,0)$ is cubic with exatly one real root in 
the interval $(-|N|-1, |N|+1)$ and by definition of $M_{1,min}$, 
$M_{1,max}$, we obtain $M_{1,min}=-|N|-1$ και $M_{1,max}=|N|+1$.
\\
For the computation of $M_{2,max}$,$M_{2,min}$, we need compute the
polynomial $H_2$ in the variables $u,v,X$ which satisfies 
$H_2(u,v_1(u),\mathcal{X}(u,v_1(u))=0$. It is the sum of
$10$ monomials terms $\deg_uH_2=2$, $\deg_vH_2=2$, $\deg_XH_2=1$. 
In analogy with what we did above, we consider
the resultant $R_2$ with respect of the variable $v$ of the polynomials
$g(u,v)$ and $H_2(u,v,X)$ which has the property that 
$R_2(u,\mathcal{X}(u,v_1(u)))=0$.
It is the product $-(N-u)^4$ with a certain polynomial $R_{20}(u,X)$ 
which is the sum of $36$ monomials of degree $2$ with respect of $u$ and 
degree $3$ with the respect to $X$. Then, necessarily, 
$R_{20}(u,\bx(u))=0$, where, for simplicity in the notation, we have
put $\bx=\mathcal{X}(u,v_1(u))$.
Differentiating this we obtain (this is equation (8.17) of 
\cite{Nikosbook})
\[
\frac{\partial R_{20}}{\partial u}(R_{20}(u,\bx(u)) +
   \frac{\partial R_{20}}{\partial \bx}(R_{20}(u,\bx(u))\cdot\bx'(u)=0
\] 
with $\bx'(u)$ meaning the derivative of $\bx(u)$ with respect to $u$.
The left-hand side is a polynomial $H_3$ in 
the variables $u,X,X'$, linear in $X'$, with the property 
$H_3(u,\bx(u),\bx'(u))=0$ identically.
This equation along with $R_{20}(u,\bx(u))=0$ suggest to 
consider the resultant, with respect of the variable $X$, of the 
polynomials $H_3(u,X,X')$ and $R_{20}(u,X)$. This we denote
by $R_3(u,X')$; it satisfies $R_3(u,\bx'(u))=0$.
According to our computations,
$R_3(u,X')$ is the product of an integer, times $(-u+N)$, times the 
square of a linear polynomial in $u$ (only) whose root belongs to the 
interval $(-|N|,|N|+1)$, times a polynomial $R_{30}(u,X')$ which is a sum of
28 monomials and $\deg_u(R_{30})=9$ and $\deg_{X'}(R_{30})=3$. Since we 
assume that $|u|>B_0=|N|+1$, then, necessarily, $R_{30}(u,\bx'(u))=0$ and 
we apply Lemma \ref{lemma 8.5.1} with $F=R_{30}$, $V=\bx'$, 
$U_0=B_0=|N|+1$.
Now $F(u,0)$ is a cubic polynomial with exactly one real root, 
approximately equal to $-|N|-1$. Therefore, in the notation of the 
aforementioned Lemma, in the present situation we have $U_{min}=-|N|-1$ 
and $U_{max}=|N|+1$. Consequently, by the definition of 
$M_{2,min}$, $M_{2,max}$, we obtain $M_{2,min}=-|N|-1$ and 
$M_{2,max}=|N|+1$.

According to Section 8.5 of \cite{Nikosbook}, this implies that 
$B_1=|N|+1$.

We perform the computation of $B_2,B_3$ and the constants 
$\th, c_9,c_{10}$ following the detailed instructions of
\cite[Section 8.6]{Nikosbook}.
First we have to compute (symbolic computation) the rational 
function $G(u,v)$ defined explicitly in 
\cite[Proposition 8.4.1]{Nikosbook}. 
Then, following Section 8.4 of \cite{Nikosbook}, we set 
$\bg(u)=G(u,v_1(u))$ (notice the difference between $\bg$ and $g$).
According to \cite[Proposition 8.4.2]{Nikosbook}, there exist constants 
$B_2\geq B_1$, $c_9>0$ and $\theta$ which satisfy 
\[
\left|\frac{\bg(u)}{\bg_v(u,v_1(u))}\right|\leq c_9|u|^{-1-\theta}
\]
($\bg_v$ means derivative with respect to $v$).
For the practical computation of these constants, a detailed example is 
discussed in \cite[Section 8.6]{Nikosbook}. 
Here we follow an analogous method. We denote by 
$\mathcal{I}=\mathcal{I}(u)$ the rational function inside the absolute 
value in the left-hand side of the above displayed inequality.
Clearing out the denominator in the last relation gives an explicit
polynomial equation $H_4(u,v_1(u),\mathcal{I})=0$. Our computations
show that 
$H_4(u,v_1(u),\mathcal{I})=-(-u+N)^6H_{40}(u,v_1(u),\mathcal{I})$,
where $H_{40}(u,v,\mathcal{I})$ is a certain polynomial in
$u,v,\mathcal{I}$. Since $|u|>|N|+1$, we must have
$H_{40}(u,v_1(u),\mathcal{I})=0$.
But we also have $g(u,v_1(s))=0$, so that we can eliminate $v_1(u)$ 
from the last two equations to obtain a relation which, according to our 
computations, is the following:
 $\mbox{constant}\cdot(N-u)^8h(u)R_4(u,\mathcal{I})=0$,
where $h(u)$ is a quartic polynomial in $u$ which, as it is easily seen, 
has no real roots for $|n|\geq2$. Then, necessarily, $R_4(u,\mathcal{I})=0$
and, following the method and notation of Section 8.6 of \cite{Nikosbook}, 
we write this equation as follows:
\begin{align}\label{pol_I}
\mathcal{I}^3+q_2(u)\mathcal{I}+q_3(u)=0,
\end{align}
with
\[
q_2(u)=\dfrac{28}{3q(u)}, \quad q_3(u)=\dfrac{8}{3q(u)},
\]
where 
\begin{align*}
q(u)=&2025u^6-4050u^4+(-4050N^3+4050N+200)u^3+2025u^2
\\&+(4050N^3-4050N-200)u-12+2025N^2+200N-200N^3-4050N^4+2025N^6.
\end{align*}
Now we work as follows.
Consider $q_2(u)$. Its numerator has no real roots and those of the 
denominator belong to the interval $(-|N|-1,|N|+1)$. But we have already 
assumed that $|u|\geq |N|+1$ so (for $|N|\geq2$) we have $q_2(u)>0$ and 
$q_3(u)>0$. 
Consequently, in \eqref{pol_I}, $\mathcal{I}<0$.
Setting $\mathcal{I}=-\mathcal{J}<0$ we obtain the equation
\begin{align}\label{pol_j}
\mathcal{J}^3+q_{2}(u)\mathcal{J}-q_3(u)=0
\end{align}
where now the strictly negative coefficient are $-q_3(u)$.
By Cauchy's rule\footnote{see ``Cauchy's rule'' in Section $8.6$ in 
\cite{Nikosbook}} ,
\[
0<\mathcal{J}<\mathrm{max}\{q_3(u)^{1/3}\}.
\]
If $|u|>3|N|$ then $q(u)>1800u^6$, implying 
$q_3(u)<\dfrac{28}{3\cdot1800}u^{-6}$. 
We obtain $0<-\mathcal{I}=\mathcal{J}<0.17|u|^{-2}$.
Hence, in the notation of \cite[Proposition 8.4.2]{Nikosbook}, 
\[
B_2=3|N|, \quad \theta=1=\dfrac{1}{\nu_s},\quad c_9=0.17. 
\]

Next we must compute constants $B_3$, $c_{10}$ and $c_{11}$ such that:
If $g(u,v)=0$ with $u$ an integer $>\max\{B_2,B_3\}$, and 
$x=\mathcal{X}(u,v)$, then 
$h(x)\leq c_{10}+c_{11}\log |u|$, where $h(x)$ denotes the absolute
logarithmic height of $x$.
For the practical computation of these constants we apply 
\cite[Proposition 8.7.1]{Nikosbook}.

We write the relation 
$g(u,v)=0$ 
in the form 
\[
v^3+a_{1}(u)v^2+a_2(u)v+a_3(u)=0
\]
where $a_1(u)=-4$, $a_2(u)=3$ and $a_3(u)=-15u^3+15u+15N^3-15N$.
Let $B_{3}$ be a constant larger than every root of every non-zero 
polynomial $a_i$. 

We easily check that we can take 
\[
B_{3}=|N|+1.
\]
Thus, in the sequel we will assume that the point $P^C=(u(P),v(P))$ 
satisfies
\begin{equation}
                \label{eq lowbd |u(P)|}
|u(P)| \geq 3|N|                
\end{equation}
and, for simplicity in the notation we put $(u(P),v(P))=(u,v)$.

Assume $u\geq3N$ and $N\geq2$.
\begin{itemize}
\item If $v\geq0$, then Cauchy's rule implies 
\[
0\leq v\leq\max\{2|a_1(u)|,\left(2|a_3(u)|\right)^{\frac{1}{3}}\}
=\max\{8,\sqrt[3]{2}\,|30u^3|^{\frac{1}{3}}\}=2\sqrt[3]{15}\,|u|.
\]
\item If $v<0$, we put $v=-w$ with $w>0$ so that $\mathrm{g}(u,v)=0$ 
is written as 
\[
w^3+b_1(u)w^2+b_2(u)w+b_3(u)=0,
\]
where $b_1(u)=4$, $b_2(u)=3$ and $b_3(u)=15u^3-15u-15N^3+15N>0,$ for 
$u\geq 3N$. So the above polynomial has no real roots.
\end{itemize}
Therefore we conclude $|v|\leq2\sqrt[3]{15}\,|u|$.

Next assume $u\leq-3N$ and $N\leq-2$. 
\\ \noindent
Then, we consider 
$\overline{g}(u,v)=g(-u,v)$ instead of $g(u,v)$.
Working as above we obtain the bound $|v|\leq2\sqrt[3]{15}\,|u|$.

Thus, in general, for $|u|\geq3|N|$ ($|N|\geq 2$) we have 
$|v|\leq2\sqrt[3]{15}\,|u|$ and, consequently, the absolute value of the 
numerator of $\mathcal{X}(u,v)$ is, easily, bounded by
\[
(180n^2+108)|u|^2+(180n^3+60n^2+60n+108)|u|+120n^3+60n+9\leq200n^3|u|^2
\]
and, clearly, $200N^3|u|^2$ is an upper bound for the absolute value of 
denominator $(N-u)^2$ of $\mathcal{X}(u,v).$ Thus
\begin{align*}
h(x(P))&=\log|x(P)|=\\
&=\log\max\{\mbox{numer}(|\mathcal{X}(u(P),v(P)|),
\mbox{denom}(|\mathcal{X}(u(P),v(P))|)\}\\
&\leq\log(200|N|^3|u(P)|^2)=\log(200|N|^3)+2\log|u(P)|
\end{align*}
and consequently, $c_{10}=\log(200|N|^3)$ and $c_{11}=2$.

Summing up, our computations furnished us with the following values:
\begin{equation}
                \label{eq the uniform constants}
B_2=3|N|,\quad \th=1,\quad c_9=0.17,\quad B_3=|N|+1, 
\quad c_{10}= \log(200|N|^3),\quad c_{11}=2.                
\end{equation}

\section{{\large The canonical height of $P_0^E$ in Section \ref{subsec d=-1}}}
        \label{hP0}
In this appendix we compute an upper bound for the canonical height of the 
point $P_0$, by applying \cite[Proposition 2.6.4]{Nikosbook}.
\begin{lemma} 
                \label{lemma hP0}
For the elliptic curve $E: y^2=x^3-1575x-12451725$ and its point $P_0^E$,
defined in \emph{Section \ref{subsec d=-1}}, we have                
\[
\hat{h}(P_0^E)\leq 7.647146073.
\]
\end{lemma}
\begin{proof}
According to Section \ref{subsec d=-1}, 
$x(P_0)=-15\zeta+63\zeta^2$, where $\zeta=\sqrt[3]{15}$.
The minimal polynomial of $x(P_0)$ is
$ x^3-360x^2+20925x-66442950$,
therefore, by \cite[Proposition 2.4.2]{Nikosbook} we have 
$h(x(P_0))=\frac{1}{3}\log(66442950)$.

The discriminant $\Delta$ and the $j$-invariant of $E$ are, respectively,
\[
\Delta=-66979386718470000, \quad j=-\frac{59270400}{9187844543}.
\]
Applying \cite[Proposition 2.6.4]{Nikosbook}, to the elliptic curve $E$
with $D=E$, we obtain
$\hat{h}(P_0^E)\leq 7.300572483$.
\end{proof}

\section{{\large Reduced upper bound of $M$ in Section \ref{subsec d=-1}}}
        \label{LLL_d_-1}
In this appendix we reduce the upper bound \eqref{BM_d_-1} of $M$
following the very explicit procedure described in the first four pages 
of \cite[Chapter 10]{Nikosbook}.
This is based on de Weger's reduction process \cite{dW} which makes use of 
the LLL-algorithm \cite{LLL} to problems of the following general type:
\emph{Let $\lambda=n_0+n_1\xi_1+\cdots +n_k\xi_k$, where the $\xi_i$'s are
explicitly known real numbers and $n_0,n_1,\ldots,n_k$ are unknown 
integers, such that $N=\max_{0\leq i\leq k}|n_i|\leq B$ with $B$ an 
explicit ``huge'' positive number and 
$|\lambda|\leq\kappa_1\exp(\kappa_2-\kappa_3N^2)$
with $\kappa_1,\kappa_2,\kappa_3$ explicit positive numbers.
Exploit this to find a considerably smaller upper bound for $N$, which is
of the size of $\log B$.
}

We keep the notations of \cite[Chapter 10]{Nikosbook}. 
In our case $\lambda$ is the linear form $L(P)$ up to a multiplicative 
constant and the relation $|\lambda|\leq\kappa_1\exp(\kappa_2-\kappa_3N^2)$
comes from \cite[Theorem 6.8]{Nikosbook}, which guarantees that, if
$|u(P)|\geq 80$, then 
$|L(P)|\leq 4a^{-1/2}\exp(0.5\log(3c_7)+\gamma-\rho M^2)$. 
We have

\begin{align}\label{lamda}
\lambda:=\lambda(P):=\dfrac{dr_0}{\zeta_1}
L(P)=n_0+n_1\xi_1+n_2\xi_2+n_3\xi_3+n_4\xi_4+n_5\xi_5
\end{align}
where 
\[
\xi_i=\frac{r_0\ell_i}{\zeta_1}\quad(i=1,\ldots,4)
\mbox{ and }\xi_5=\frac{r_0\ell_0}{\zeta_1},
\]
where $\ell_i=\ellog(P_i),\,i=1,\ldots,5$.
In the notation of \cite[Chapter 10]{Nikosbook} 
we have $k=6$, $d=1$, $r_0=1$, $\alpha=5/2$ and $\beta=3/2$
and $N=\frac{5}{2}M+\frac{3}{2}$. Therefore
\[ 
N\leq 2.6M, \mbox{ if $M\geq15$, hence $M^2\geq(2.6)^{-2}N^2$,}
\]
so, in the notation of \cite[Chapter 10]{Nikosbook},
$\kappa_3=(2.6)^{-2}=0.1479$ and, by \cite[Relation (10.3)]{Nikosbook} 
$\kappa_1=0.3458142306$, $\kappa_2=8.318175470$, 
 and $\kappa_4=0.06077760153$.

Choice of $C$: 
According to \eqref{BM_d_-1}, $M\leq 6.86\cdot10^{147}$, therefore 
a first upper bound for $N$ is $B_1(N):=6.86\cdot10^{147}$ and,
according to \cite[Relation (10.7)]{Nikosbook}, the integer $C$ must be 
somewhat larger than  
\[
2^{k(k+1)/2}\left(k+\frac{1}{2}\right)^{k+1}\! B_1(N)^{k+1}
=2^{21}\cdot 6.5^7\cdot\left(6.86\cdot10^{147}\right)^7
                    \lessapprox  10^{1046}.
\]
We choose $C=10^{1050}$ and work with precision $1080$ decimal digits. 
The linear form $\lambda$ to which we apply the reduction process is
\begin{align*}
\lambda=\frac{1}{\zeta_1}L(P)= &
n_0+n_1\left(\frac{\ell_1}{\zeta_1}\right)
+n_2\left(\frac{\ell_2}{\zeta_1}\right)
+n_3\left(\frac{\ell_3}{\zeta_1}\right)
+n_4\left(\frac{\ell_4}{\zeta_1}\right)
+n_5\left(\frac{\ell_5}{\zeta_1}\right) \\
& = n_0+n_1\xi_1+n_2\xi_2+n_3\xi_3+n_4\xi_4+n_5\xi_5 \\
& = n_0+n_1(-48478...)+n_2(-254638...)+n_3(-11376...)+n_4(-17120...)
+n_5(38222...).
\end{align*}
The lattice $\Gamma$ which is generated by the columns of the matrix
\[
\mathcal{M}_{\Gamma}=
\begin{pmatrix}
1 & 0 & 0 & 0 & 0 & 0 \\ 
0 & 1 & 0 & 0 & 0 & 0 \\ 
0 & 0 & 1 & 0 & 0 & 0 \\
0 & 0 & 0 & 1 & 0 & 0 \\
0 & 0 & 0 & 0 & 1 & 0 \\
\left[\C'\xi_1\right] & \left[\C'\xi_2\right] & \left[\C'\xi_3\right] 
& \left[\C'\xi_4\right] & \left[\C'\xi_5\right] & \C'
\end{pmatrix}
\]
is a sublattice of $\Z^6$, where
\[
\left[\C'\xi_1\right] =\underbrace{-4847823699\ldots3498111567}_{1080 
\mbox{ \scriptsize digits }},
\quad
\left[\C'\xi_2\right] =\underbrace{-2546386009\ldots1065528645}_{1080 
\mbox{ \scriptsize digits }},
\]
\[
\left[\C'\xi_3\right] =\underbrace{-1137615354\ldots6305151167}
_{1080\mbox{ \scriptsize digits }},
\quad
\left[\C'\xi_4\right] =\underbrace{-1712017639\ldots6800307853}
_{1080\mbox{ \scriptsize digits }},
\]
\[
\left[\C'\xi_5\right] =\underbrace{3822275389\ldots0614559892}
_{1080\mbox{ \scriptsize digits }}.
\]
All six integer coordinates of the first vector $\textbf{b}_0$ of the
$\mathrm{LLL}$-reduced basis have 180 digits
and the length of $\textbf{b}_0$ is of the size of $7.85\cdot10^{179}$,
satisfying thus the relation 
\begin{equation}
               \label{eq condition for |b0|}
 |\textbf{b}_0|>2^{k/2}(k+\frac{1}{2})B_1(N)               
\end{equation}
(cf.~\cite[Relation (10.6)]{Nikosbook}). It follows then by
\cite[Proposition 10.1.1]{Nikosbook} that
\begin{equation}
                \label{eq new ub for N}
\cu_4N^2 \leq \cu_2+\log(\cu_1C)
    -\log\{\sqrt{2^{-k}|\textbf{b}_0|^2-k B_1(N)}-k B_1(N) \}                
\end{equation}
from which we obtain $N\leq186$.

We set now $B_1(N)=186$ and repeat the process, by choosing 
$\C'=10^{30}$. 
We obtain
\[
\mbox{ new }\textbf{b}_0=\begin{pmatrix*}[r]
-53853\\15304\\ -25937\\ 245\\-36760\\12425
\end{pmatrix*}. 
\]
The new $\textbf{b}_0$ satisfies \eqref{eq condition for |b0|}
hence, from \eqref{eq new ub for N} we obtain the new upper bound 
$N\leq33$.
Repeating the process we obtain the new upper bound $N\leq 27$,
which cannot be further reduced.

{\bf Important computational issue}.
In \eqref{eq new ub for N} the parameter $\kappa_4$ is equal to
an explicitly calculable multiple of $\rho$, the least eigenvalue of 
the (positive definite) height-pairing matrix;
this is detailed in the beginning of Chapter 10 of \cite{Nikosbook}.
It is clear then that, the smaller $\rho$ is, the larger is the
upper bound for $N$ which is obtained from \eqref{eq new ub for N}.
This shows that, as the reduction process goes on and $C$ becomes 
smaller and smaller,
the role of $\rho$ becomes more and more important: The larger is $\rho$ 
the smaller will be the reduced upper bound for $N$. Therefore, it is 
important to compute a Mordell-Weil basis whose height-pairing matrix
has its least eigenvalue as small as possible. We start from a Mordell-Weil
basis furnished by {\sc magma} and then follow the algorithm of 
Stroeker \& Tzanakis \cite[Section 4]{Stro Tza}, which we implemented
in {\sc maple}. In our case, the online {\sc magma} calculator 
({\tt V2.24-3}) furnished us with the basis
\[
(235,395),\;
(750,-20205),\;
(310,-4105),\;
(495,10395),\;
(1075,35045).
\]
and corresponding $\rho\approx 0.410937$. Using the above mentioned 
algorithm we obtained the Mordell-Weil basis that we use in 
Section \ref{subsec d=-1}. As explained in the ``Important remark'' of
that section, just above Table \ref{Table All integer points_d_-1}, 
by using the improved basis a lot of computation time is gained.
 
\end{appendices}
%

\end{document}